\documentclass[12pt]{amsart}
\usepackage[mathscr]{eucal}
\usepackage{amsmath}
\usepackage{ulem}
\usepackage{color}

\usepackage[pagewise]{lineno}

\theoremstyle{plain}
	\newtheorem{theorem}{Theorem}[section]
	\newtheorem{lemma}[theorem]{Lemma}
	\newtheorem{corollary}[theorem]{Corollary}
	
	\newtheorem{proposition}[theorem]{Proposition}
	\newtheorem{remark}[theorem]{Remark}

\theoremstyle{plain}

\def\R{\mathbb{R}}

\def\calC{\mathcal{C}}

\def\calZ{\mathcal{Z}}

\def\e{\varepsilon}
\def\6{\partial}
\def\8{\infty}
\def\tu{\tilde{u}}

\def\tz{\tilde{z}}
\def\tw{\tilde{w}}

\def\ttn{\tilde{t}_n}
\def\tz{\tilde{z}}

\def\bu{\bar{u}}

\makeatletter
\newcommand{\xRightarrow}[2][]{%
\ext@arrow 0055{\Rightarrowfill@}{#1}{#2}%
}
\def\Rightarrowfill@{\arrowfill@\Relbar\Relbar\Rightarrow}
\newcommand{\xLeftarrow}[2][]{%
\ext@arrow 0055{\Leftarrowfill@}{#1}{#2}%
}
\def\Leftarrowfill@{\arrowfill@\Leftarrow\Relbar\Relbar}
\newcommand{\xLongleftrightarrow}[2][]{%
\ext@arrow 0055{\llrafill@}{#1}{#2}%
}
\def\llrafill@{\arrowfill@\Leftarrow\Relbar\Rightarrow}
\makeatother

\makeatletter
 \@addtoreset{equation}{section}
\makeatother

\begin{document}

\title[Single point rupture solutions]{Radial single point rupture solutions\\ for a general MEMS model\\}

\author{Marius Ghergu}
\thanks{}
\address{School of Mathematics and Statistics, University College Dublin, Belfield, Dublin 4, Ireland}
\address{Institute of Mathematics Simion Stoilow of the Romanian Academy, 21 Calea Grivitei St., 010702 Bucharest, Romania}
\email{marius.ghergu@ucd.ie}

\author{Yasuhito Miyamoto}
\thanks{The second author was supported by JSPS KAKENHI Grant Numbers 19H01797, 19H05599.}
\address{Graduate School of Mathematical Sciences, The University of Tokyo, 3-8-1 Komaba, Meguro-ku, Tokyo 153-8914, Japan}
\email{miyamoto@ms.u-tokyo.ac.jp}

\begin{abstract}
We study the initial value problem
$$
\begin{cases}
r^{-(\gamma-1)}\left(r^{\alpha}|u'|^{\beta-1}u'\right)'=\frac{1}{f(u)} & \textrm{for}\ 0<r<r_0,\\
u(r)>0 & \textrm{for}\ 0<r<r_0,\\
u(0)=0,
\end{cases}
$$
for $\gamma>\alpha>\beta\geq 1$ and  $f\in C[0,\bar u)\cap C^2(0,\bar u)$, $f(0)=0$, $f(u)>0$ on $(0, \bar u)$ and $f$ satisfies certain assumptions which include the standard case of pure power nonlinearities encountered in the study of Micro-Electromechanical Systems (MEMS).
We obtain the existence and uniqueness of a solution $u^*$ to the above problem, the rate at which it approaches the value zero at the origin and the intersection number of points with the corresponding regular solutions $u(\,\cdot\,,a)$ (with $u(0,a)=a$) as $a\to 0$.

In particular, these results yield the uniqueness of a radial single point rupture solution and other qualitative properties for MEMS models. The bifurcation diagram is also investigated.

\end{abstract}

\date{\today}
\maketitle

\noindent
\section{Introduction and Main results}
In this paper we are concerned with the following initial value problem
\begin{equation}\label{S1E1}
\begin{cases}
r^{-(\gamma-1)}\left(r^{\alpha}|u'|^{\beta-1}u'\right)'=\frac{1}{f(u)} & \textrm{for}\ 0<r<r_0,\\
u(r)>0 & \textrm{for}\ 0<r<r_0,\\
u(0)=0,
\end{cases}
\end{equation}
where $\gamma>\alpha>\beta\geq 1$ and for some $\bar u>0$ we have
$$
f\in C^2(0,\bar u)\cap C[0,\bar u),\quad f(0)=0\quad \textrm{and}\ \ f>0\ \textrm{on}\ (0,\bar u).
$$
We also assume that $g(u)=f(u)^{1/\beta}$ satisfies:
\begin{enumerate}
\item[(G1)] $
g'(u)>0,\ \ g''(u)>0\ \ \textrm{for all}\ u\in(0,\bu).$

\item[(G2)] There exists $\displaystyle q:=\lim_{u\to 0}\frac{g'(u)^2}{g(u)g''(u)}$.
\item[(G3)] If $q=1$, then $\displaystyle\frac{g'(u)^2}{g(u)g''(u)}\ge 1$ for small $u>0$.
\end{enumerate}
We shall see that $q\ge 1$ and there exists
\begin{equation}\label{S1E3}
L:=\lim_{u\to 0}\frac{g'(u)G(u)}{g(u)^2}\quad\textrm{and}\quad L\in (\frac{1}{2},1],
\end{equation}
where $G(u)=\int_0^ug(s)ds$.

Note that the following choice of $\alpha$, $\beta$, $\gamma$ cover the standard setting of Laplace, $p$-Laplace and $k$-Hessian operators as follows:
\medskip

\begin{center}
\begin{tabular}{lccc}
 & $\alpha$ & $\beta$ & $\gamma$\\
 \hline\hline
        Laplace $\quad\qquad$& $N-1$\; & $1$ & \;\;$N$\\
     
$p$-Laplace $\quad\qquad$ & $N-1$\; & \;$p-1$\; & \;\;$N$\\

$k$-Hessian $\quad\qquad$ & $N-k$\; & $k$ & \;\;$N$\\
\hline\hline
\end{tabular}
\end{center}
\medskip

Our study of \eqref{S1E1} is motivated by the elliptic problem
\begin{equation}\label{S1Em}
\begin{cases}
-\Delta v=\frac{\lambda|x|^\sigma}{(1-v)^m} & \textrm{in}\ \Omega,\\
v=0 & \textrm{on}\ \partial \Omega,\\
0<v\leq 1 & \textrm{in}\ \Omega,
\end{cases}
\end{equation}
which arises in the mathematical modelling of Micro-Electromechanical Systems (MEMS), a technology that designs various types of microscopic devices by combining electronic units with micro-size mechanical components. In a simplified setting, such devices consist of two parallel plates of which one is fixed at level $v=1$ and has zero voltage while the other plate is a thin dielectric membrane with fixed boundary at level $v=0$. A voltage directly proportional to $\lambda>0$ is applied  and the membrane deflects towards the rigid plate. If the voltage exceeds a certain critical value $\bar{\lambda}$ called a pull-in voltage, then an instability occurs in the process. For an introduction into MEMS modelling, the reader is referred to the monograph \cite{PB03}.  A mathematical account on this topic is provided in \cite{EGG10}. Similar problems to \eqref{S1Em} which involve different differential operators are discused in \cite{CES08} (for the $p$-Laplace operator) and in \cite{CFT11, LW14, LW17} (for the biharmonic operator). It is easy to see that if $v$ is a radial solution of \eqref{S1Em} in $\Omega=B_1$, then $u(\sqrt{\lambda}r):=1-v(r)$ satisfies (\ref{S1E1}) with $f(u)=u^m$, $\alpha=N-1$, $\beta=1$, $\gamma=N+\sigma$ and $r_0=\sqrt{\lambda}$.  

One key research topic in relation to \eqref{S1Em} is the study of {\it rupture solutions}, that is, positive solutions of \eqref{S1Em} for which the rupture set defined as $\{v=1\}$ is nonempty. Rupture solutions for \eqref{S1Em} are investigated in \cite{DW12, DWW16,GW08b, GW08c, GW14, LGG05}. In \cite{DWW16} the authors study the rupture set and its Hausdorf dimension for MEMS equation in dimension $N\geq 2$. The rupture solutions of a MEMS equation with fringing field is investigated in \cite{DW12}. The authors in \cite{LGG05} obtained the existence of a single point rupture solution for
$$
\begin{cases}
\Delta u =\frac{1}{f(u)}\,,\ u>0 & \quad\mbox{ in }B_R\setminus\{0\},\\
u(0)=0,&
\end{cases}
$$
where $f\in C^1(0, \infty)$ satisfies
\begin{equation}\label{S1Ee}
\int_u^\infty \frac{ds}{f(s)}\leq \frac{c u}{f(u)}\quad\mbox{ for all }0<u<\delta,
\end{equation}
where $c, \delta>0$. Let us observe that \eqref{S1Ee} is stronger than conditions $(G1)$-$(G3)$ above; for instance $f(u)=e^{-1/u^m}$, $m>0$, and $f(u)=\exp\{-e^{1/u}\}$ satisfy $(G1)$-$(G3)$ (see Examples 4 and 5 in the next section) but not \eqref{S1Ee}.

The parabolic counterpart of \eqref{S1Em} namely
\begin{equation}\label{S1Ep}
\begin{cases}
\partial_t v-\Delta v=\frac{\lambda|x|^\sigma}{(1-v)^m} & \textrm{in }\ \Omega\times (0, T),\\
v=0 & \textrm{on }\ \partial \Omega\times (0, T),\\
v(x,0)=0&  \textrm{in }\ \Omega,
\end{cases}
\end{equation}
has been considered in \cite{ES18, ES19, GS15, GPW05}. In such a setting, solutions $v$ of \eqref{S1Ep} which take the value 1 at some point inside of $\Omega$ are called {\it touchdown solutions}.

For any $a>0$ denote by $u(r,a)$ the solution of
\begin{equation}\label{S1E4}
\begin{cases}
r^{-(\gamma-1)}\left(r^{\alpha}|u'|^{\beta-1}u'\right)'=\frac{1}{f(u)} & \textrm{for}\ 0<r<r_0,\\
u(r)>0 & \textrm{for}\ 0<r<r_0,\\
u(0)=a
\end{cases}
\end{equation}
and we call $u(r,a)$ a regular solution of (\ref{S1E4}). The existence of regular solutions $u(\,\cdot\,, a)\in C^2(0, r_0)\cap C[0, r_0)$, for $r_0>0$ small, can be achieved using the same techniques as in \cite[Proposition 2.1]{M16}. 

In the present work we shall be interested in qualitative properties of solutions to \eqref{S1E1}. We obtain that for some small $r_0>0$, problem \eqref{S1E1} has a unique solution $u^*$ which is the  limit of regular solutions $u(\,\cdot\,, a)$ of \eqref{S1E4} as $a\to 0$. Uniqueness, regularity and behavior  at the origin of the solution $u^*$ as well as the intersection number of points between $u^*$ and $u(\,\cdot\,, a)$ are also discussed. In particular, for specific values of exponents $\alpha, \beta$ and $\gamma$ we deduce the uniqueness of single point rupture solution to \eqref{S1Em} as well as to its $p$-Laplace counterpart (see Section 2 below for more relevant examples). As remarked in \cite{EG08}, uniqueness is not an easy task even for regular (i.e., non-rupture) solutions to \eqref{S1Em}.

For a continuous function $h(r)$ defined on the interval $I\subset \R$, we introduce the zero number of $h$ on $I$ by
\[
\calZ_I[h(\,\cdot\,)]=\sharp\{r\in I\ |\ h(r)=0\}.
\]
Our main result is stated below.
\begin{theorem}\label{THA}
Assume $\gamma>\alpha>\beta\ge 1$.
Then, the following hold: 
\begin{enumerate}
\item[(i)] There exists a unique solution $u^*$ of (\ref{S1E1}).
Moreover, $u^*$ satisfies:
\begin{equation}\label{E5}
u^*(r)=G^{-1}\left[\frac{r^{\theta}}{A}(1+o(1))\right]\ \ \textrm{as}\ \ r\to 0,
\end{equation}
where
\begin{equation}\label{E6}
\theta=1+\frac{\gamma-\alpha}{\beta}>1,
\end{equation}
and
\begin{equation}\label{E7}
A=\theta\left\{\gamma-L(\gamma-\alpha+\beta)\right\}^{1/\beta}>0.
\end{equation}
\item[(ii)] $u^*$ is of class $C^1$ at origin if and only if $\lim_{s\to 0}\frac{ s^{\theta}}{G(s)}$ is finite.
In this case,
\begin{equation}\label{S1E8}
u'(0)=\left[\frac{1}{A} \lim_{s\to 0}\frac{ s^{\theta}}{G(s)}\right]^{1/\theta}.
\end{equation}
In particular, if 
\begin{equation}\label{S1E9}
\frac{\gamma-\alpha}{\beta}<
\begin{cases}
\frac{q}{q-1} & \textrm{if}\ q>1,\\
\infty & \textrm{if}\ q=1,
\end{cases}
\end{equation}
then, $u^*$ is not of class $C^1$ at the origin. 
\item[(iii)] We have
\begin{equation}\label{S1Econv}
u(r,a)\to u^*\ \ \textrm{in}\ \ C^2_{loc}(0,r_0)\cap C_{loc}[0,r_0)\ \ \textrm{as}\ \ a\to 0,
\end{equation}
where $u(r,a)$ is the regular solution of (\ref{S1E4}).
\item[(iv)] 
Let
\begin{equation}\label{qc}
q_c:=\frac{1}{2}+\frac{\theta(\beta+1)^2}{4(\sqrt{\tau}-\sqrt{\theta})^2-2\theta(\beta+1)^2},
\end{equation}
where $\tau:=\beta\gamma+\alpha-\beta>1$.
If $q<q_c$, then, for each $\rho\in(0,r_0)$, we have
$$
\calZ_{(0,\rho)}\left[u(\,\cdot\,,a)-u^*(\,\cdot\,)\right]\to\infty\ \ \textrm{as}\ \ a\to 0.
$$
\end{enumerate}
\end{theorem}

\begin{remark}
(i) Since $q\geq 1$ (see Lemma \ref{S3L1}~(i)) we see that condition $q_c>q$ in Theorem~\ref{THA}~(iv) implies in particular $q_c>1$ which is equivalent to $\alpha-\beta<4\theta$.\\
(ii) In the Laplace case, that is $\alpha=N-1$, $\beta=1$, $\gamma=N$, we have
$$q_c=\frac{1}{2}+\frac{1}{N-2-2\sqrt{N-1}}.$$
Hence, $\gamma>\alpha>\beta\ge 1$ and $q_c>1$ hold if and only if $2<N<10$.
\end{remark}

We define $g_q(u)$ and $G_q(u)$ as follows:
$$
g_q(u)=
\begin{cases}
e^u & \textrm{if}\ q=1,\\
u^p,\ \frac{1}{p}+\frac{1}{q}=1, & \textrm{if}\ q>1,
\end{cases}
\qquad
G_q(u)=
\begin{cases}
e^u & \textrm{if}\ q=1,\\
\frac{u^{p+1}}{p+1},\ \frac{1}{p}+\frac{1}{q}=1, & \textrm{if}\ q>1.
\end{cases}
$$
Let us consider the problem
\begin{equation}\label{S1E10}
\begin{cases}
s^{-(\gamma-1)}\left(s^{\alpha}|v'|^{\beta-1}v'\right)'=\frac{1}{g_q(v)^{\beta}}, & s>0,\\
v(0)=b>0,\\
v'(0)=0.
\end{cases}
\end{equation}
Then
\begin{equation}\label{S1E11}
v^*(s)=G_q^{-1}\left[A^{-1}s^{\theta}\right]
\end{equation}
is an exact solution of the equation in (\ref{S1E10}), where $A=\theta\{\gamma-L(\gamma-\alpha+\beta)\}^{1/\theta}$, $\theta=1+(\gamma-\alpha)/\beta$ and $L=q/(2q-1)$.
Note that $v^*(s)$ vanishes at the origin only if $q>1$.
If $q=1$, then $v^*(0)=-\infty$ and we will use
$g_1(v)$ and $G_1(v)$ in our approach to prove Theorem~\ref{THA}~(iv).

As an application of Theorem~\ref{THA} we consider the following bifurcation problem:
\begin{equation}\label{v}
\begin{cases}
s^{-(\gamma-1)}(s^{\alpha}|v'|^{\beta-1}v')'+\frac{\lambda}{f(1-v)}=0 & \textrm{for}\ 0<s<1,\\
v(s)>0 & \textrm{for}\ 0<s<1,\\
v(1)=0,
\end{cases}
\end{equation}
where $\lambda>0$.
Let $g=f^{1/\beta}$.
Here, in addition to (G1)--(G3) we assume the following:
\medskip

\begin{enumerate}
\item[(G4)] $g'(u)>0\ \ \textrm{for}\ \ 0<u<\infty$.
\end{enumerate}
\medskip

The nonlinear term $\lambda/f(1-v)$ has a singularity at $v=1$.
If $0\le v<1$ in $[0,1]$, then the solution $v$ is regular.
It is known that the set of the regular solutions $\calC:=\{(\lambda,v)\}$ can be parametrized by $\tau:=v(0)$, and that $\calC$ emanates from $(0,0)$.
By the results in \cite{K97}, $\calC$ can be expressed as $\calC:=\{(\lambda(\tau),v(s,\tau))\}$.
Using Theorem \ref{THA} we obtain the existence and uniqueness of a pair $(\bar{\lambda}, \bar{v})$ such that $\bar{v}$ is the (unique) solution of \eqref{v} with $\lambda=\bar{\lambda}$ and $\bar{v}(0)=1$. We shall call $\bar{v}$ a degenerate solution of \eqref{v}.  More, precisely, we have:
\begin{corollary}\label{C0}
The following hold:
\begin{enumerate}
\item[(i)] The problem (\ref{v}) has a unique degenerate solution $(\bar{\lambda},\bar{v})$.
Let 
$$
\calC:=\{(\lambda(\tau),v(s,\tau))|\ v(0, \tau)=\tau,\ 0<\tau<1\}
$$ 
be the bifurcation curve.
Then, as $\tau\to 1$ one has
\begin{equation}\label{C0E0}
\lambda(\tau)\to\bar{\lambda}\quad\textrm{and}\quad v(s,\tau)\to\bar{v}(s)\ \textrm{in}\ C[0,1].
\end{equation}
\item[(ii)] Let $q_c$ be defined by (\ref{qc}).
If $q<q_c$, then the curve $\calC$ has infinitely many turning points around $\bar{\lambda}$, and hence (\ref{v}) with $\lambda=\bar{\lambda}$ has infinitely many regular solutions.
\end{enumerate}
\end{corollary}

The authors in \cite{GW08a} obtain the existence of infinitely many turning points for \eqref{S1Em} by using a technique developed in \cite{D00} (see also \cite{GLWZ11}). The $p$-Laplace version of \eqref{S1Em} was discused in \cite{K18} where a similar conclusion to that in our Corollary~\ref{C0} above was obtained except the uniqueness of the rupture solution. 
We also refer the interested reader to \cite{LW11} for the detailed asymptotic behavior of the bifurcation curve near $\tau=1$ and numerical experiments.
Quasilinear equations including the $k$-Hessian operator were considered in \cite{dd17}.

Turning to our study of (\ref{S1E1}), let us briefly describe the main tools of our approach.
The solution $u^*$ of (\ref{S1E1}) is constructed as the limit of the regular solutions $u(r,a)$ of (\ref{S1E4}) for $a\to 0$.
The uniqueness of solution $u^*$ requires a detailed ODE analysis of (\ref{S1E1}) near the origin, in which a major role is played by the change of variables
\[
e^{z(t)}=\frac{G[u^*(r)]}{G_q[v^*(r)]}\ \ \textrm{and}\ \ t:=-\log r,
\]
where $v^*$ is given by \eqref{S1E11}. This is a generalization of the Emden transformation that corresponds to the case $f(u)=u^p$.
The construction of the solution to \eqref{S1E1} does not rely on the contraction mapping theorem as mentioned above but rather on a limiting process.
The detailed ODE analysis we develop in Section~4 yields $z(t)\to 0$ and $z_t(t)\to 0$ as $t\to \infty$ which further allows us to derive the uniqueness of solution $u^*$ of \eqref{S1E1} and the leading term of its expansion around the origin.

Our method is applicable to the study of the singular solution $u^*(r)$ of supercritical problems of type $-\Delta u=f(u)$, where $u^*(r)\to\infty$ as $r\to 0$.
See \cite{GG20,M18a,MN18,MN20} for the counterpart of Theorem~\ref{THA} in such a setting.

Finally, in the study of the number of intersection points between the solution $u^*$ of \eqref{S1E1} and regular solutions $u(\cdot, a)$ of \eqref{S1E4}, 
the following transformation plays a key role:
\begin{equation}\label{quasi}
G_q[\tu(s)]=\lambda^{-\theta}G[u(r,a)]\ \ \textrm{where}\ \ s:=\frac{r}{\lambda}\ \ \textrm{ and }\ \ \lambda:=\left(\frac{G[a]}{G_q[1]}\right)^{1/\theta}. 
\end{equation}
Although the original equation (\ref{S1E1}) does not have a scaling invariance, through the limiting process given by (\ref{quasi}) we are led to (\ref{S1E10}) whose structure of solution set is known. In this way, the analysis of (\ref{S1E1}) can be reduced to that of (\ref{S1E10}).

The remaining of the paper is organised as follows. In order to show the full strength of Theorem \ref{THA} we provide some relevant examples of nonlinearities $f(u)$ and determine the explicit behavior around the origin of the  solution $u^*$. This is done in Section~2. In Section~3 we prove the existence of the solution to (\ref{S1E1}) while in Section~4 we prove the uniqueness and the convergence property (\ref{S1Econv}).
In Section~5 we study the intersection properties using a blow-up argument and prove Theorem \ref{THA} (iv).
In Section~6 we prove Corollary~\ref{C0}.


\section{Examples}

In this section we illustrate the results in Theorem \ref{THA} for some specific nonlinearities $f(u)$. 
For two functions $h_1(u)$, $h_2(u)$ defined in a positive neighborhood of the origin, by $h_1(u)\asymp h_2(u)$ we understand that $h_1(u)/h_2(u)\to 1$ as $u\to 0^+$.
\subsection{Example 1}
Consider problem \eqref{S1E1} with $\gamma>\alpha>\beta\geq 1$ and let $f(u)$ be given by
\begin{equation}\label{f1}
f(u):=
\left\{
\begin{aligned}
&u^m|\log u-C|^d &&\quad\mbox{ if } u>0,\\
&0&&\quad\mbox{ if } u=0,
\end{aligned}
\right.
\end{equation}
where $m>\beta$, $d\geq 0$ and $C\in\R$. We have $g(u)=f(u)^{1/\beta}$ and $q=m/(m-\beta)>1$. To compute the asymptotic behavior in \eqref{E5} we use L'Hospital's rule and find
$$
G(u)\asymp \frac{\beta}{m+\beta} u^{\frac{m+\beta}{\beta}} |\log u|^{\frac{d}{\beta}}.
$$
Using the above asymptotic behavior of $G(u)$ we have
$$
\lim_{u\to 0}\frac{G^{-1}(u)}{\Big( \frac{u^\beta}{|\log u|^{d}}\Big)^{\frac{1}{m+\beta}}}=\lim_{v=G^{-1}(u)\to 0} \frac{v}{\Big( \frac{G(v)^\beta}{|\log G(v)|^{d}}\Big)^{\frac{1}{m+\beta}}}=
\Big(\frac{m+\beta}{\beta}\Big)^{\frac{d+\beta}{m+\beta}}.
$$
Hence,
$$
G^{-1}(u)\asymp \left[ C\frac{u^\beta}{|\log u|^d}\right]^{\frac{1}{m+\beta}}\quad\mbox{ where }C=
\Big(\frac{m+\beta}{\beta}\Big)^{d+\beta}.
$$
Using Theorem~\ref{THA} we find:

\begin{proposition}\label{p1} 
Let $\gamma>\alpha>\beta\geq 1$ and let $f$ be defined by \eqref{f1} where $m>\beta$, $d\geq 0$ and $C\in \R$. Then, there exists $r_0>0$ and a unique solution $u^*$ of
\eqref{S1E1}.

Furthermore, 
$$
u^*(r)=\kappa \left[\frac{r^{\gamma-\alpha+\beta}}{|\log r|^d}  \right]^{\frac{1}{m+\beta}}\big(1+o(1)\big)\quad\mbox{ as }r\to 0,
$$
where
$$
\kappa=\left[ \frac{m+\beta}{\beta\gamma+m(\alpha-\beta)} \Big(\frac{m+\beta}{\gamma-\alpha+\beta}\Big)^{d+\beta}\right]^{\frac{1}{m+\beta}}>0.
$$
Also, $u^*$ is of class $C^1$ at the origin if and only if $\gamma\geq m+\alpha$.

If $u(\,\cdot\,, a)$ denotes the regular solution of \eqref{S1E4} with $f(u)$ defined in \eqref{f1}, then:
\begin{enumerate}
\item[(i)] $u(\,\cdot\,, a)\to u^*$ in $C^2_{\rm{loc}}(0,r_0)\cap C_{\rm{loc}}[0, r_0)$ as $a\to 0$;
\item[(ii)] Let $q_c$ be defined by (\ref{qc}), i.e.,
$$
q_c=\frac{1}{2}+\frac{(\gamma-\alpha+\beta)(\beta+1)^2}{4(\sqrt{\beta(\beta\gamma+\alpha-\beta)}-\sqrt{\gamma-\alpha+\beta})^2-2(\gamma-\alpha+\beta)(\beta+1)^2}.
$$
If $\frac{m}{m-\beta}<q_c$, then for any $\rho\in(0,r_0)$ one has
$$
\calZ_{(0,\rho)}\left[ u(\,\cdot\,,a)-u^*(\,\cdot\,)\right]\to\infty\ \ \textrm{as}\ \ a\to 0.
$$

\end{enumerate}
\end{proposition}
\subsection{Example 2.} As a particular case of Proposition \ref{p1} let us consider the following MEMS problem for the $p$-Laplace operator:
\begin{equation}\label{MEMSp}
\begin{cases}
\; -\Delta_p v=\frac{|x|^\sigma}{(1-v)^m}\,, v<1  \textrm{ in } B_{r_0}\setminus\{0\}\subset \R^N,\\
\; v(0)=1,
\end{cases}
\end{equation}

where $N> p\geq 2$, $m>p-1$, $\sigma\ge 0$. If $v$ is any radially symmetric solution of \eqref{MEMSp}, then $u=1-v$ is a solution of \eqref{S1E1} with 
$$
\gamma=N+\sigma\,,\; \alpha=N-1\,,\;\beta=p-1\,\mbox{ and }f(u)=u^m,
$$
so Proposition \ref{p1} applies in this case. By direct calculations we also have that
\begin{equation}\label{vstar}
v^*(r)=1-\left[ \frac{m+p-1}{(N+\sigma)(p-1)+(N-p)m} \Big(\frac{m+p-1}{p+\sigma}   \Big)^{p-1} \right]^{\frac{1}{m+p-1}} r^{\frac{p+\sigma}{m+p-1}} 
\end{equation}
is a solution of \eqref{MEMSp}.  Our result regarding \eqref{MEMSp} is stated below.

\begin{proposition}
Let $N>p\geq 2$, $m>p-1$ and $\sigma\ge 0$. 
\begin{enumerate}
\item[(i)] There exists $r_0>0$ such that $v^*$ defined by \eqref{vstar} is the unique radial solution of \eqref{MEMSp};
\item[(ii)] For any $a\in (0,1)$ let $v(\,\cdot\,, a)$ be the radial solution of  
$$
\begin{cases}
\; -\Delta_p v(\,\cdot\,, a)=\frac{|x|^\sigma}{(1-v(\,\cdot\,, a))^m}\;,\;\; 0<v<1  \;\; \mbox{ in } B_{r_0},\\
\; v(0, a)=a.
\end{cases}
$$
Then:
\begin{enumerate}
\item[(ii1)] $v(\,\cdot\,, a)\to v^*$ in $C^2_{\rm{loc}}(0,r_0)\cap C_{\rm{loc}}[0, r_0)$ as $a\to 0$;
\item[(ii2)] Let
$$
q_c=\frac{1}{2}+\frac{p^2(p+\sigma)}{4(\sqrt{(p-1)\{\sigma(p-1)+p(N-1)\}}-\sqrt{p+\sigma})^2-2p^2(p+\sigma)}.
$$
If $\frac{m}{m-p+1}<q_c$, then for any $\rho\in(0,r_0)$ one has
$$
\calZ_{(0,\rho)}\left[ v(\,\cdot\,,a)-v^*(\,\cdot\,)\right]\to\infty\ \ \textrm{as}\ \ a\to 0.
$$
\end{enumerate}
\end{enumerate}
\end{proposition}

\subsection{Example 3} Let $1\leq k<N/2$ be a positive integer and denote by 
$S_k(D^2 u)$ the $k$-Hessian operator of a radially symmetric function $u=u(r)$, that is,  
$$
S_k(D^2 u)=\frac{1}{k}\binom{N-1}{k-1} r^{-(N-1)}\Big[ r^{N-k} (u')^{k} \Big]'.
$$
Consider  the problem
\begin{equation}\label{khes}
\begin{cases}
\; c_{N,k}S_k(D^2 u) =\frac{1}{u^m}\,,\; u>0  \;\textrm{ in } B_{r_0}\setminus\{0\}\subset \R^N,\\
\; u(0)=0,
\end{cases}
\end{equation}
where $m>k$ and $c_{N,k}=k\binom{N-1}{k-1}^{-1}>0$. 
As a consequence of Proposition \ref{p1} we find:

\begin{proposition}
Let $1\leq k<\min\{m, N/2\}$ be a positive integer.
\begin{enumerate}
\item[(i)] There exists $r_0>0$ and a unique increasing radially symmetric solution $u^*$ of \eqref{khes}. Furthermore,
$$
u^*(r)=\left[ \frac{m+k}{Nk + m(N-2k)} \Big(\frac{m+k}{2k}\Big)^{k}\right]^{\frac{1}{m+k}}
r^{\frac{2k}{m+k}}.
$$
\item[(ii)] For any $a\in (0,1)$ let $u(\,\cdot\,, a)$ be the radial solution of  
$$
\begin{cases}
\; c_{N,k}S_k(D^2 u(\,\cdot\,, a)) =\frac{1}{u(\,\cdot\,, a)^m}\,,\; u(\,\cdot\,, a)>0  \;\textrm{ in } B_{r_0},\\
\; u(0, a)=a.
\end{cases}
$$
Then:
\begin{enumerate}
\item[(ii1)] $u(\,\cdot\,, a)\to v^*$ in $C^2_{\rm{loc}}(0,r_0)\cap C_{\rm{loc}}[0, r_0)$ as $a\to 0$;
\item[(ii2)] Let
$$
q_c=\frac{1}{2}+\frac{(k+1)^2}{2(\sqrt{Nk+N-2k}-\sqrt{2})^2-2(k+1)^2}.
$$
If $\frac{m}{m-k}<q_c$, then for any $\rho\in(0,r_0)$ one has
$$
\calZ_{(0,\rho)}\left[ u(\,\cdot\,,a)-u^*(\,\cdot\,)\right]\to\infty\ \ \textrm{as}\ \ a\to 0.
$$
\end{enumerate}
\end{enumerate}
\end{proposition}

\subsection{Example 4}
Let
\begin{equation}\label{f2}
f(u):=
\left\{
\begin{aligned}
&e^{-\frac{1}{u^m}} &&\quad\mbox{ if } u>0,\\
&0&&\quad\mbox{ if } u=0,
\end{aligned}
\right.
\end{equation}
where $m>0$. We have $q=1$ and $g:=f^{1/\beta}$ satisfies $g'(u)^2/g(u)g''(u)\ge 1$ for small $u>0$. 
Thus, by Lemma~\ref{S3L1}~(iii) condition (G3) holds.
By L'Hospital's rule we obtain
$$
G(u)\asymp \frac{\beta}{m} u^{m+1} e^{-\frac{1}{\beta u^m}}.
$$
From here we deduce
$$
\lim_{u\to 0}\frac{G^{-1}(u)}{\big(\log \frac{1}{u}\big)^{-1/m}}=
\lim_{v=G^{-1}(u)\to 0} \frac{v}{\big( \log \frac{1}{G(v)}\big)^{-1/m}}=\beta^{-\frac{1}{m}}.
$$
Thus, $G^{-1}(u)\asymp |\beta \log u|^{-{1}/{m}}$.
Using Theorem~\ref{THA} we find:

\begin{proposition} Let $\gamma>\alpha>\beta\geq 1$ and let $f$ be defined by \eqref{f2} where $m>0$.  Then, there exists $r_0>0$ and a unique solution $u^*$ of
\eqref{S1E1}.

Furthermore, 
$$
u^*(r)=|(\gamma-\alpha+\beta)\log r|^{-\frac{1}{m}} \big(1+o(1)\big)\ \ \mbox{as}\ \ r\to 0.
$$
If $u(\,\cdot\,, a)$ denotes the regular solution of \eqref{S1E4} with $f(u)$ defined in \eqref{f2}, then
\begin{enumerate}
\item[(i)] $u(\,\cdot\,, a)\to u^*$ in $C^2_{\rm{loc}}(0,r_0)\cap C_{\rm{loc}}[0, r_0)$ as $a\to 0$;
\item[(ii)] If $4\gamma>(\alpha-\beta)(\beta+4)$, then any $\rho\in (0,r_0)$ we have 
$$
{\mathcal Z}_{(0,\rho)}[u(\,\cdot\,, a)-u^*(\,\cdot\,)]\to \infty\ \ \mbox{as}\ \ a\to 0.
$$
\end{enumerate}
\end{proposition}

\subsection{Example 5}
Let
\begin{equation}\label{f3}
f(u):=
\left\{
\begin{aligned}
&\exp\{-e^{1/u}\} &&\quad\mbox{ if } u>0,\\
&0&&\quad\mbox{ if } u=0,
\end{aligned}
\right.
\end{equation}
where $\exp\{v\}=e^v$. An easy calculation yields $q=1$ and the function $g:=f^{1/\beta}$ satisfies $g'(u)^2/(g(u)g''(u))\ge 1$ for small $u>0$.
By Lemma~\ref{S3L1}~(iii), condition (G3) holds.
By L'Hospital's rule we find
$$
G(u)\asymp \beta u^2 e^{-1/u}g(u).
$$
Further, we compute 
$$
\lim_{u\to 0}\frac{G^{-1}(u)}{\frac{1}{\log(|\log u|)}} =
\lim_{v=G^{-1}(u)\to 0} \frac{v}{\frac{1}{\log(|\log G(v)|)}}=1.
$$
Hence, $G^{-1}(u)\asymp \frac{1}{\log(|\log u|)}$.
Using Theorem~\ref{THA} we find:

\begin{proposition} Let $\gamma>\alpha>\beta\geq 1$ and let $f$ be defined by \eqref{f3}.  Then, there exists $r_0>0$ and a unique solution $u^*$ of
\eqref{S1E1}.

Furthermore, 
$$
u^*(r)=\frac{1+o(1)}{\log(|\log r|)}\ \ \mbox{as}\ \ r\to 0.
$$
If $u(\,\cdot\,, a)$ denotes the regular solution of \eqref{S1E4}  with $f(u)$ defined in \eqref{f3}, then
\begin{enumerate}
\item[(i)] $u(\,\cdot\,, a)\to u^*$ in $C^2_{\rm{loc}}(0,r_0)\cap C_{\rm{loc}}[0, r_0)$ as $a\to 0$;
\item[(ii)] If $4\gamma>(\alpha-\beta)(\beta+4)$, then and any $\rho\in (0,r_0)$  we have 
$$
{\mathcal Z}_{(0,\rho)}[u(\,\cdot\,, a)-u^*(\,\cdot\,)]\to \infty\ \ \mbox{as}\ \  a\to 0.
$$
\end{enumerate}
\end{proposition}

\section{Existence}
\begin{lemma}\label{S3L1}
The following hold:
\begin{enumerate}
\item[(i)] The limit  $q$ defined in (G2) satisfies $q\ge 1$.
\item[(ii)] The limit $L$ defined in (\ref{S1E3}) exists and $L=q/(2q-1)$. In particular $L\in(\frac{1}{2},1]$.
\item[(iii)] If $\frac{g'(u)^2}{g(u)g''(u)}\ge 1$, then (G3) holds, that is, $\frac{g'(u)G(u)}{g(u)^2}\le 1$.
\end{enumerate}
\end{lemma}
\begin{proof}
(i) Assume by contradiction that there exist $u_0\in (0, \bar u)$ and $q_0\in (0,1)$ such that
\[
\frac{g'(u)^2}{g(u)g''(u)}\le q_0<1\ \ \textrm{for all}\ \ 0<u\le u_0.
\]
Then $u\longmapsto g'(u)g(u)^{-1/q_0}$ is increasing over $(0, u_0]$. Hence,
$$
g'(u)g(u)^{-1/q_0}\leq g'(u_0)g(u_0)^{-1/q_0}\ \ \textrm{for all}\ \ 0<u\le u_0.
$$
Integrating both sides of this inequality over $[u,u_0]$ we have 
\[
\frac{q_0}{1-q_0}\left(g(u)^{1-\frac{1}{q_0}}-g(u_0)^{1-\frac{1}{q_0}}\right)
\le\frac{g'(u_0)}{g(u_0)^{1/q_0}}(u_0-u)
\ \ \textrm{for all}\ \ 0<u\le u_0.
\]
Letting $u\to 0$, we obtain a contradiction, since the left hand side diverges as $u\to\ 0$.
Thus if the limit $q$ exists, then $q\ge 1$.\\
(ii) We show that $\lim_{u\to 0}g(u)^2/g'(u)=0$.
Using (G1) we find that $g'$ is increasing on $(0,\bu)$ so $g(u)=\int_0^ug'(s)ds\le ug'(u)$ for all $0<u<\bu$.
Then,
\[
0\le\lim_{u\to\infty}\frac{g(u)^2}{g'(u)}\le\lim_{u\to 0}\frac{g(u)ug'(u)}{g'(u)}=0.
\]
Applying the L'Hospital's rule, we have
\[
L=\lim_{u\to 0}\frac{G[u]}{g(u)^2/g'(u)}=\lim_{u\to 0}\frac{g(u)}{2g(u)-g(u)^2g''(u)/g'(u)^2}=\frac{1}{2-1/q}.
\]
Hence, $L=q/(2q-1)$.
Since $q\ge 1$, we see that $1/2<L\le 1$.\\
(iii) By (G3) we have $g(u)\le 2g(u)-g(u)^2g''(u)/g'(u)^2$.
Integrating it over $[0,u]$, by (ii) we have
\[
G[u]\le\left[\frac{g(u)^2}{g'(u)}\right]_0^u=\frac{g(u)^2}{g'(u)}\ \ \textrm{for small}\ u>0.
\]
The conclusion follows.
\end{proof}

The next result provides basic properties of regular solutions $u(r,a)$.
\begin{lemma}\label{S3L2}
Let $u(r,a)$ be a solution of (\ref{S1E4}). Then
\begin{enumerate}
\item[(i)] $u(r,a)$ satisfies:
$$
u'(r,a)>0\ \ \textrm{and}\ \ r^{\alpha}u'(r,a)^{\beta}=\int_0^r\frac{s^{\gamma-1}}{f(u(s,a))}ds
\ \ \textrm{for $r>0$ small.}
$$
\item[(ii)] $u(r,a)$ is nondecreasing in $r$ and
$$
G(u(r,a))-G(a)\ge\frac{r^{\theta}}{\theta\gamma^{1/\beta}}\quad\textrm{and}\quad
u(r,a)\ge G^{-1}\left[\frac{r^{\theta}}{\theta\gamma^{1/\beta}}+G(a)\right].
$$
\end{enumerate}
\end{lemma}
\begin{proof}
From (\ref{S1E4}) we have
\begin{equation}\label{S3L2E1}
\left(r^{\alpha}|u'|^{\beta-1}u'\right)'=\frac{r^{\gamma-1}}{f(u)},\quad r>0.
\end{equation}
Since $u(0,a)=a>0$, $f(u(r,a))$ is bounded away from zero for $r>0$ small so from (\ref{S3L2E1}) one gets $u'(r,a)>0$ and
$$
r^{\alpha}u'(r,a)^{\beta}=\int_0^r\frac{s^{\gamma-1}}{f(u(s,a))}ds.
$$
This proves (i).\\
Since $s\mapsto f(u(s,a))$ is increasing, for $r>0$ small we have
$$
r^{\alpha}u'(r,a)^{\beta}\ge\frac{1}{f(u(r,a))}\int_0^rs^{\gamma-1}ds=\frac{r^{\gamma}}{\gamma g(u(r,a))^{\beta}}.
$$
so that
$$
u'(r,a)g(u(r,a))\ge\frac{r^{\frac{\gamma-\alpha}{\beta}}}{\gamma^{1/\beta}}
=\frac{r^{\theta-1}}{\gamma^{1/\beta}}\ \ \textrm{for}\ r>0.
$$
Integrating the above inequality we deduce (ii).
\end{proof}

\begin{lemma}\label{S3L3}
The following hold:
\begin{enumerate}
\item[(i)] $u(r,a)$ satisfies $0<u'(r,a)\le\kappa u(r,a)/r$ for $r>0$ small,
where $\kappa=\theta(\frac{\gamma}{\alpha-\beta})^{1/\beta}>0$.
\item[(ii)] Let $\delta>0$ small and:
\begin{equation}\label{S3L3E9}
\delta_0:=\frac{\delta}{2}\left(1+\frac{\beta\kappa}{\alpha-\beta}\right)^{-1},\qquad
r_{\delta}:=\left[\frac{\delta}{2}\theta(\gamma f(\delta_0))^{1/\beta}\right]^{1/\theta}.
\end{equation}
Then, if $0<a<\delta_0$ one has
$$
u(r,a)<\delta\ \ \textrm{for all}\ \ 0<r<r_{\delta}.
$$
\end{enumerate}
\end{lemma}
\begin{proof}
(i) We claim first that $r\mapsto\frac{G(u(r,a)}{g(u(r,a))}$ is increasing for $r>0$ small.
Indeed, by Lemma~\ref{S3L1} (iii) and $u'(r,a)\ge 0$ we have
$$
\frac{d}{dr}\left[\frac{G(u(r,a))}{g(u(r,a))}\right]=\left(1-\frac{G(u(r,a))g'(u(r,a))}{g(u(r,a))^2}\right) u'(r,a)\ge 0.
$$
Next, using Lemma~\ref{S3L2}~(ii) we further estimate
\begin{equation}\label{S3L3E10}
r^{\alpha}u'(r,a)^{\beta}=\int_0^r\left[\frac{G(u(s,a))}{g(u(s,a))}\right]^{\beta}\frac{s^{\gamma-1}}{G(u(r,a))^{\beta}}ds
\le \left[\frac{G(u(r,a))}{g(u(r,a))}\right]^{\beta}\int_0^r\frac{s^{\gamma-1}}{G(u(r,a))^{\beta}}ds.
\end{equation}
From Lemma~\ref{S3L2}~(ii) one has $G(u(r,a))\ge\frac{r^{\theta}}{\theta\gamma^{1/\beta}}$ which yields
\begin{equation}\label{S3L3E11}
G(u(s,a))^{\beta}\ge\frac{s^{\gamma-\alpha+\beta}}{\theta^{\beta}\gamma}\ \ \textrm{for $s>0$ small.}
\end{equation}
Also, since $g$ is increasing in a small neighborhood of the origin we have
\begin{equation}\label{S3L3E12}
G(u(r,a))\le g(u(r,a))u(r,a)\ \ \textrm{for $r>0$ small}.
\end{equation}
Using (\ref{S3L3E11}) and (\ref{S3L3E12}) in estimate (\ref{S3L3E10}) we further deduce
\begin{multline*}
r^{\alpha}u'(r,a)^{\beta}\le u(r,a)^{\beta}\int_0^r\frac{s^{\gamma-1}}{G(u(s,a))^{\beta}}ds
\le\theta^{\beta}\gamma u(r,a)^{\beta}\int_0^rs^{\alpha-\beta-1}ds\\
=\frac{\theta^{\beta}\gamma}{\alpha-\beta}r^{\alpha-\beta}u(r,a)^{\beta}\ \ \textrm{for $r>0$ small}.
\end{multline*}
This yields
$$
u'(r,a)\le\theta\left(\frac{\gamma}{\alpha-\beta}\right)^{1/\beta}\frac{u(r,a)}{r}
\ \ \textrm{for $r>0$ small}.
$$
(ii) Assume by contradiction that there exist $a\in (0,\delta_0)$ and $r_*\in (0,r_{\delta})$ so that
\begin{equation}\label{S3L3E13}
u(r,a)<\delta\ \ \textrm{on}\ \ (0,r_*)\ \ \textrm{and}\ \ u(r_*)=\delta.
\end{equation}
Since $u$ is increasing on $(0,r_*)$, there exists a unique $r_1\in (0,r_*)$ so that
\begin{equation}\label{S3L3E14}
u(r_1,a)=\delta_0\ \ \textrm{and}\ \ u(r,a)>\delta_0\ \ \textrm{for all}\ r\in (r_1,r_*].
\end{equation}
Consider the problem
\begin{equation}\label{S3L3E15}
\begin{cases}
\left(r^{\alpha}v'(r)^{\beta}\right)'=\frac{r^{\gamma-1}}{f(\delta_0)} & \textrm{on}\ (r_1,r_*),\\
v(r_1)=\delta_0=u(r_1,a),\\
v'(r_1)=u'(r_1,a).
\end{cases}
\end{equation}
From (\ref{S1E4}) and (\ref{S3L3E15}) one has
$$
r^{\alpha}\left(v'(r)^{\beta}-u'(r,a)^{\beta}\right)'
=r^{\gamma-1}\left(\frac{1}{f(\delta_0)}-\frac{1}{f(u(r,a))}\right)>0\ \ \textrm{on}\ \ (r_1,r_*),
$$
since, by (\ref{S3L3E14}) one has $u(\,\cdot\,,a)>\delta_0$ on $(r_1,r_*]$.
By integration in the above inequality it follows that $v'>u'(\,\cdot\,,a)$ on $(r_1,r_*]$.
Since $v(r_1)=u(r_1,a)=\delta_0$, by (\ref{S3L3E15}) one further obtains $v\ge u(\,\cdot\,,a)$ on $[r_1,r_*]$.

Let us now return to the main equation of (\ref{S3L3E15}) which by integration yields
$$
r^{\alpha}v'(r)^{\beta}-r_1^{\alpha}v'(r_1)^{\beta}
=\int_{r_1}^r\frac{s^{\gamma-1}}{f(\delta_0)}ds
<\frac{r^{\gamma}}{\gamma f(\delta_0)}\ \ \textrm{for all}\ \ r_1\le r\le r_*.
$$
Hence,
$$
v'(r)\le\left(\frac{r_1}{r}\right)^{\alpha/\beta}v'(r_1)+\frac{1}{\left(\gamma f(\delta_0)\right)^{1/\beta}}r^{\theta-1}\ \ \textrm{for all}\ \ r_1\le r\le r_*.
$$
Integrating the above inequality over $[r_1,r_*]$ implies
$$
v(r_*)-v(r_1)\le\frac{\beta}{\alpha-\beta}r_1^{\alpha/\beta}v'(r_1)
\left( r_1^{1-\frac{\alpha}{\beta}}-r_*^{1-\frac{\alpha}{\beta}}\right)
+\frac{1}{\theta\left(\gamma f(\delta_0)\right)^{1/\beta}}r_*^{\theta}.
$$
Thus,
\begin{equation}\label{S3L3E16}
v(r_*)<\frac{\beta}{\alpha-\beta}r_1v'(r_1)
+\frac{1}{\theta\left(\gamma f(\delta_0)\right)^{1/\beta}}r_*^{\theta}
+v(r_1).
\end{equation}
By (\ref{S3L3E14}), (\ref{S3L3E15}) and part (i) above one has
\begin{equation}\label{S3L3E17}
r_1v'(r_1)=r_1u'(r_1,a)\le\kappa u(r_1,a)=\kappa\delta_0.
\end{equation}
Also from (\ref{S3L3E9}) one has
\begin{equation}\label{S3L3E18}
\delta_0\left(1+\frac{\beta\kappa}{\alpha-\beta}\right)=\frac{\delta}{2}
\quad\textrm{and}\quad
r_*^{\theta}<r_{\delta}^{\theta}=\frac{\delta}{2}\theta(\gamma f(\delta_0))^{1/\beta}.
\end{equation}
We combine (\ref{S3L3E13}), (\ref{S3L3E16}), (\ref{S3L3E17}), (\ref{S3L3E18}) and $v\ge u(\,\cdot\,,a)$ to get
$$
\delta=u(r_*,a)\le v(r_*)<\delta_0\left(1+\frac{\beta\kappa}{\alpha-\beta}\right)+\frac{1}{\theta(\gamma f(\delta_0))^{1/\beta}}r_{\delta}^{\theta}=\delta,
$$
which is a contradiction.
This concludes the proof of part (ii).
\end{proof}

\begin{proof}[Proof of Theorem 1(Existence)]
Take $\{a_n\}\subset (0,\infty)$ a decreasing sequence of positive real numbers that converges to zero.
By Lemma~\ref{S3L3}~(ii) one may find $r_0>0$ such that $\{u(\,\cdot\,,a_n)\}$ is uniformly bounded on any compact interval $I\subset (0,r_0)$.
By the estimate in Lemma~\ref{S3L2}~(ii) one has that $\{u(\,\cdot\,,a_n)\}$ is bounded away from zero on any compact interval $I\subset (0,r_0)$.
Also, integrating in (\ref{S1E4}) one has
\begin{equation}\label{S3L3E19}
r^{\alpha}u'(r,a_n)^{\beta}=\rho u'(\rho,a_n)^{\beta}
+\int_{\rho}^r\frac{s^{\gamma-1}}{f(u(s,a_n))}ds
\ \ \textrm{for all}\ 0<\rho< r<r_0.
\end{equation}
This shows that $\{ u'(\,\cdot\,,a_n)\}$ is uniformly bounded on any compact interval $I\subset (0,r_0)$, since it follows from Lemma~\ref{S3L3}~(i) that $u'(\rho,a_n)$ is bounded uniformly in $n$.
Differentiating in (\ref{S3L3E19}) it follows that $\{u''(\,\cdot\,,a_n)\}$ and $\{u'''(\,\cdot\,,a_n)\}$ are uniformly bounded on any compact interval $I\subset (0,r_0)$.
Now, using a diagonal process and Arzel\`{a}-Ascoli theorem, there exists a subsequence of $\{a_n\}$, which is still denoted by $\{a_n\}$ in the following, and $u^*\in C^2(0,r_0)\cap C[0,r_0)$ such that
\begin{equation}\label{S3L3E20}
\lim_{n\to\infty}u(\,\cdot\,,a_n)=u^*\ \ \textrm{in}\ \ C^2_{loc}(0,r_0)\cap C_{loc}[0,r_0).
\end{equation}
By Lemma~\ref{S3L2}~(ii) it follows that $u^*>0$ in $(0,r_0)$ and by Lemma~\ref{S3L3}~(ii) one has $u^*(0)=\lim_{r\to 0}u^*(r)=0$.
Thus, $u^*$ is a solution of (\ref{S1E1}).
\end{proof}

\section{Uniqueness of degenerate solution and convergence}
\begin{lemma}\label{S4L1}
Let $u(r)$ be a solution of (\ref{S1E1}). The following hold:
\begin{enumerate}
\item[(i)] $u'>0$ in $(0,r_0)$ and
$$
r^{\alpha}u'(r)^{\beta}=\int_0^r\frac{s^{\gamma-1}}{f(u(s))}ds\ \ \textrm{for all}\ \ r\in(0,r_0).
$$
\item[(ii)] For $r>0$ small we have
$$
G(u(r))\ge\frac{r^{\theta}}{\theta\gamma^{1/\beta}}\ \ \textrm{and hence}\ \ 
u\ge G^{-1}\left[\frac{r^{\theta}}{\theta\gamma^{1/\beta}}\right].
$$
\end{enumerate}
\end{lemma}


\begin{proof}
(i) From (\ref{S1E1}) it follows that $r\mapsto r^{\alpha}|u'(r)|^{\beta-1}u'(r)$ is increasing, and hence there exists $\ell:=\lim_{r\to 0}r^{\alpha}|u'(r)|^{\beta-1}u'(r)\in [-\infty,\infty)$.
If $\ell<0$ then $u'<0$ in a neighborhood of zero.
This further implies $u(r)<u(0)=0$ for $r>0$ small which is impossible.
If $\ell>0$ then one can find $\e>0$ and $\rho\in (0,a)$ so that 
$$
r^{\alpha}u'(r)^{\beta}\ge \e\ \ \textrm{for all}\ \ 0<r\le\rho.
$$
By integration one gets 
$$
u(\rho)\ge u(r)+\e^{1/\beta}\int_r^{\rho}s^{-\alpha/\beta}ds\to\infty\ \ \textrm{as}\ \ r\to 0,
$$
again contradiction.
Hence $\ell=\lim_{r\to 0}r^{\alpha}|u'(r)|^{\beta-1}u'(r)=0$.

Using this fact and integrating in (\ref{S1E1}) we find
$$
r^{\alpha}|u'(r)|^{\beta-1}u'(r)=\int_0^r\frac{s^{\gamma-1}}{f(u(s))}ds\ \ \textrm{for all}\ \ 0<r<r_0.
$$
This yields $u'>0$ and concludes the proof of part (i).\\
(ii) Since $u$ is increasing on $(0,r_0)$ it follows that $s\mapsto f(u(s))$ is increasing for small $s>0$.
Thus, from (i) above we find
$$
r^{\alpha}u'(r)^{\beta}\ge\frac{1}{f(u(r))}\int_0^rs^{\gamma-1}ds=\frac{r^{\gamma}}{\gamma f(u(r)))}
\ \ \textrm{for $r>0$ small}.
$$
Then
$$
u'(r)g(u(r))\ge\frac{r^{\theta-1}}{\gamma^{1/\beta}}\ \ \textrm{for $r>0$ small}.
$$
Integrate in the above inequality to further deduce (ii).
\end{proof}

\begin{lemma}\label{S4L3}
Any degenerate solution $u$ of (\ref{S1E1}) satisfies
\begin{equation}\label{S4L3E22}
\frac{1}{\theta\gamma^{1/\beta}}\le\liminf_{r\to 0}\frac{G(u(r))}{r^{\theta}}<\infty.
\end{equation}
\end{lemma}
\begin{proof}
Since $r\mapsto g(u(r))$ is increasing for $r>0$ small, we have
$$
r^{\alpha}u'(r)^{\beta}
=\int_0^r\frac{s^{\gamma-1}}{f(u(s))}ds
\ge\frac{1}{g(u(r))^{\beta}}\int_0^r s^{\gamma-1}ds
=\frac{1}{\gamma g(u(r))^{\beta}}r^{\gamma}
\ \ \textrm{for $r>0$ small.}
$$
This yields
$$
u'g(u)\ge\frac{1}{\gamma^{1/\beta}}r^{\theta-1}
\ \ \textrm{for $r>0$ small.}
$$
Integrate the above estimate to get
$$
G(u(r))\ge\frac{1}{\theta\gamma^{1/\beta}}r^{\theta}
\ \ \textrm{for $r>0$ small}
$$
and this proves the lower bound in (\ref{S4L3E22}).
To prove the upper bound in (\ref{S4L3E22}) we argue by contradiction and assume $\liminf_{r\to 0}\frac{G(u(r))}{r^{\theta}}=\infty$.
This means
\begin{equation}\label{S4L3E24}
\lim_{r\to 0}\frac{G(u(r))}{r^{\theta}}=\infty.
\end{equation}
To raise a contradiction, we proceed in $3$ steps.\\
\underline{Step 1}: $r\mapsto\frac{G(u(r))}{g(u(r))}$ is increasing for $r>0$ small.\\
This argument has already been encountered when we studied regular solutions.
Using (G3) and $u'>0$ we have
$$
\frac{d}{dr}\left(\frac{G(u(r))}{g(u(r))}\right)
=\left[1-\frac{g'(u(r))G(u(r))}{g(u(r))^2}\right]u'(r)\ge 0
\ \ \textrm{for $r>0$ small.}
$$
\underline{Step 2}: There exists $\e\in(0,1)$ and $c>0$ so that $G(u(r))\ge cr^{\e}$ for $r>0$ small.\\
Take any $\e\in (0,\theta-1)$.
Using (\ref{S4L3E24}) we can find $R\in(0,r_0)$ such that
$$
\frac{G(u(r))}{r^{\theta}}>\frac{1}{(\alpha-\beta)^{1/\beta}\e}
\ \ \textrm{for all}\ 0<r\le R.
$$
Using this last estimate and Step 1, from Lemma~\ref{S4L1}~(i) we have
\begin{multline*}
r^{\alpha}u'(r)^{\beta}=\int_0^r\left[\frac{G(u(s))}{g(u(s))}\right]^{\beta}\left[\frac{s^{\theta}}{G(u(s))}\right]^{\beta}s^{\gamma-\theta\beta-1}ds\\
\le\e^{\beta}(\alpha-\beta)\left[\frac{G(u(r))}{g(u(r))}\right]^{\beta}\int_0^rs^{\alpha-\beta-1}ds
=\e^{\beta}\left[\frac{G(u(r))}{g(u(r))}\right]^{\beta}r^{\alpha-\beta}\ \ \textrm{for $r>0$ small.}
\end{multline*}
This yields
$$
\frac{u'(r)g(u(r))}{G(u(r))}\le\frac{\e}{r}\ \ \textrm{for all}\ 0<r\le R,
$$
so that
$$
\frac{d}{dr}\left[\log\frac{G(u(r))}{r^{\e}}\right]\le 0\ \ \textrm{for all}\ 0<r\le R.
$$
Integrating over $[r,R]$ one obtains
\begin{equation}\label{S4L3E25}
\frac{G(u(r))}{r^{\e}}\ge c:=\frac{G(u(R))}{R^{\e}}>0\ \ \textrm{for all}\ 0<r\le R.
\end{equation}
\underline{Step 3}: Conclusion of the proof.\\
From (\ref{S4L3E25}) we find, using $G(u)\le ug(u)$, that
$$
\frac{1}{f(u(r))}=\frac{1}{g(u(r))^{\beta}}\le 
\left[\frac{u(r)}{cr^{\e}}\right]^{\beta}
\ \ \textrm{for}\ 0<r\le R.
$$
Then,
$$
r^{\alpha}u'(r)^{\beta}
=\int_0^r\frac{s^{\gamma-1}}{f(u(s))}ds
\le \frac{1}{c^{\beta}}	\int_0^r s^{\gamma-\beta\e-1}u(s)^{\beta}ds
\le \frac{r^{\gamma-\e\beta}u(r)^{\beta}}{c^{\beta}(\gamma-\e\beta)}\ \ \textrm{for}\ 0<r\le R.
$$
This implies
$$
\frac{u'(r)}{u(r)}\le C r^{\theta-1-\e}\ \ \textrm{for}\ 0<r\le R,
$$
for some $C>0$.
Integrate the above inequality over $[r, R]$ to deduce
$$
\log\frac{u(R)}{u(r)}
\le\frac{C}{\theta-\e}\left[R^{\theta-\e}-r^{\theta-\e}\right]
\le\frac{CR^{\theta-\e}}{\theta-\e}
\ \ \textrm{for all}\ 0<r\le R.
$$
Letting $r\to 0$ in the above estimate we raise a contradiction, since $u(0)=0$.
Hence (\ref{S4L3E22}) holds and this completes the proof of Lemma~\ref{S4L3}.
\end{proof}
Let $u$ be a solution of (\ref{S1E1}).
We next introduce the scaling
\begin{equation}\label{S4E26}
e^z=\frac{A}{r^{\theta}}G(u(r))\quad\textrm{where}\quad t=-\log r\in (-\log r_0,\infty)
\end{equation}
and $\theta=1+\frac{\gamma-\alpha}{\beta}\ge 2$ and $A=\theta\{\gamma-L(\gamma-\alpha+\beta)\}^{1/\beta}$ are given by equations (\ref{E6}) and (\ref{E7}).
\begin{lemma}
We have
\begin{equation}\label{S4L3+E27}
\theta-\theta\left(\frac{\gamma}{\alpha-\beta}\right)^{1/\beta}\le z_t
\le\theta\ \ \textrm{for large}\ t>0.
\end{equation}
\end{lemma}
\begin{proof}
Note that from (\ref{S4E26}) we have
$$
z(t)=\log A+\theta t+\log G(u(r))\ \ \textrm{for large}\ t>-\log r_0.
$$
Using the fact that $\frac{d}{dt}=-r\frac{d}{dr}$ we differentiate in the above equality to deduce
\begin{equation}\label{S4L3+E28}
z_t(t)=\theta-\frac{rg(u(r))u'(r)}{G(u(r))}\le\theta
\ \ \textrm{for large}\ t>0.
\end{equation}
On the other hand, by (\ref{S4L3E22}) we have
$$
\frac{G(u(r))}{r^{\theta}}\ge\frac{1}{\theta\gamma^{1/\beta}}
\ \ \textrm{for $r>0$ large.}
$$
Using this estimate and the fact that $r\mapsto\frac{G(u(r))}{g(u(r))}$ is nondecreasing for small $r>0$, from Lemma~\ref{S4L1}~(i) we have
\begin{align*}
r^{\alpha}u'(r)^{\beta}
&=\int_0^r\left[\frac{G(u(s))}{g(u(s))}\right]^{\beta}\left[\frac{s^{\theta}}{G(u(s))}\right]^{\beta}s^{\gamma-\theta\beta-1}ds\\
&\le\left[\frac{G(u(r))}{g(u(r)}\right]^{\beta}\frac{\theta^{\beta}\gamma}{\alpha-\beta}r^{\alpha-\beta}
\ \ \textrm{for $r>0$ small.}
\end{align*}
Hence
$$
\frac{r u'(r)g(u(r))}{G(u(r))}\le\theta\left(\frac{\gamma}{\alpha-\beta}\right)^{1/\beta}
\ \ \textrm{for $r>0$ small.}
$$
Using this estimate in (\ref{S4L3+E28}) we deduce the lower bound in (\ref{S4L3+E27}).
\end{proof}

\begin{lemma}\label{S4L4}
$z$ satisfies
\begin{equation}\label{S4L4E29}
z_{tt}-az_t+b(1-e^{-\beta z})+(1-L)z_t^2+T(z)=0,
\end{equation}
where
\begin{equation}\label{S4L4E30}
a=2\theta(1-L)+\frac{\alpha}{\beta}-1>0,\qquad
b=\theta\left\{\theta(1-L)+\frac{\alpha}{\beta}-1\right\}>0,
\end{equation}
\begin{equation}\label{S4L4E32}
T(z)=\left[L-\frac{g'(u)G(u)}{g(u)^2}\right](z_t-\theta)^2
+b\left[ 1-\left(1-\frac{z_t}{\theta}\right)^{1-\beta}\right] e^{-\beta z}.
\end{equation}
\end{lemma}
\begin{proof}
Since $r=e^{-t}$ we may write (\ref{S4E26}) as $G(u(r))=\frac{1}{A}e^{z-\theta t}$.
Since $\frac{d}{dr}=-e^t\frac{d}{dt}$, we differentiate in the above equality to deduce $g(u)u'=-\frac{e^t}{A}\frac{d}{dt}\left( e^{z-\theta t}\right)$, that is,
\begin{equation}\label{S4L4E32+}
g(u)u'=-\frac{e^{z+(1-\theta)t}}{A}(z_t-\theta).
\end{equation}
We differentiate in the last equality and obtain
\begin{equation}\label{S4L4E33}
g'(u)u'^2+g(u)u''
=\frac{e^{z+(2-\theta)t}}{A}\left\{z_{tt}+(z_t-\theta)(z_t-\theta+1)\right\}.
\end{equation}
From (\ref{S4L4E32+}) we find
\begin{equation}\label{S4L4E34}
g'(u)u'^2=\frac{g'(u)}{A^2g(u)^2}e^{2z+2(1-\theta)t}(z_t-\theta)^2
=\frac{(z_t-\theta)^2}{A}\frac{g'(u)G(u)}{g(u)^2}e^{z+(2-\theta)t}.
\end{equation}
Use (\ref{S4L4E34}) in (\ref{S4L4E33}) to deduce
\begin{equation}\label{S4L4E35}
g(u)u''=\frac{e^{z+(2-\theta)t}}{A}\left\{z_{tt}+(z_t-\theta)(z_t-\theta+1)
-(z_t-\theta)^2\frac{g'(u)G(u)}{g(u)^2}\right\}.
\end{equation}
Since $f(u)\left( r^{\alpha}u'(r)^{\beta}\right)'=r^{\gamma-1}$, we find
\begin{equation}\label{S4L4E36}
(u'g)^{\beta-1}\left(\beta gu''+\frac{\alpha}{r} u'g\right)-r^{\gamma-\alpha-1}=0.
\end{equation}
Using now the expressions of $g(u)u'$ and $g(u)u''$ given by (\ref{S4L4E32+}) and respectively (\ref{S4L4E35}), from (\ref{S4L4E36}) we derive (\ref{S4L4E29}).
\end{proof}
Crucial to our approach is the following result which establishes the behaviour at infinity for $z$ and $z_t$.
\begin{proposition}\label{S4P5}
Let $z$ be defined by (\ref{S4E26}) which also satisfies (\ref{S4L4E29}).
Then
$$
\lim_{t\to\infty}z(t)=\lim_{t\to\infty}z_t(t)=0.
$$
\end{proposition}
We shall distinguish the cases where $z_t$ does not change sign or oscillates in a neighborhood of infinity.
\begin{lemma}\label{S4L6}
Assume $z_t$ does not change sign for $t>0$ large.
Then $\lim_{t\to\infty}z(t)=0$.
\end{lemma}
\begin{proof}
From Lemma~\ref{S4L3} one has
$$
\frac{1}{\theta\gamma^{1/\beta}}\le\liminf_{r\to 0}\frac{G(u(r))}{r^{\theta}}<\infty.
$$
Using (\ref{S4E26}) this implies $\liminf_{t\to\infty}z(t)\in\R$.
Since $z_t$ does not change sign for $t>0$ large it follows that there  exists $c:=\lim_{t\to\infty}z(t)$ and by the above argument we have $c\in\R$.

We claim that $c=0$.
Suppose to the contrary that this is not true.
Thus, $c\neq 0$.
Then, we prove that
\begin{equation}\label{S4L6E36}
\lim_{t\to\infty}z_t(t)=0.
\end{equation}
Since $t$ is bounded, clearly one has
\begin{equation}\label{S4L6E37}
\liminf_{t\to\infty}|z_t(t)|=0.
\end{equation}
In order to prove (\ref{S4L6E36}) it is enough to show that
\begin{equation}\label{S4L6E38}
\limsup_{t\to\infty}|z_t|=0.
\end{equation}

Suppose by contradiction that (\ref{S4L6E36}) and thus (\ref{S4L6E38}) do not hold.
Using (\ref{S4L6E37}) we may find a sequence of local extremum points $\{t_n\}\subset (0,\infty)$ of $z_t$ so that $t_n\to\infty$ and $z(t_n)\to c$, $z_t(t_n)\to 0$ and $z_{tt}(t_n)=0$.
Now from (\ref{S4L4E29}) it follows that
$$
-az_t(t_n)+b(1-e^{-\beta z(t_n)})+(1-L)z_t(t_n)^2+T(z(t_n))=0.
$$
Passing to the limit in the above equality with $n\to\infty$ (note that $T(z(t_n))\to0 $ by (G2)) one has $b(1-e^{\beta c})=0$ which is impossible since $b>0$ and $c\neq 0$.
This establishes (\ref{S4L6E36}).
Using this fact we may pass to the limit in (\ref{S4L4E29}) this time with $t\to\infty$.
It follows that
$$
\lim_{z\to\infty}z_{tt}(t)=b(e^{-\beta c}-1)\neq 0,
$$
and hence $\lim_{t\to\infty}|z_t(t)|=\infty$ which contradicts the fact that $z_t$ is bounded by (\ref{S4L6E38}).
\end{proof}
\begin{lemma}\label{S4L7}
Assume $z_t$ does not change sign for large $t$.
Then $\lim_{t\to\infty}z_t(t)=0$.
\end{lemma}
\begin{proof}
From the previous lemma we know that $\lim_{t\to\infty}z(t)=0$ and since $z$ is bounded, one has
\begin{equation}\label{S4L7E39}
\liminf_{t\to\infty}|z_t(t)|=0.
\end{equation}
Assume by contradiction that
\begin{equation}\label{S4L7E40}
\limsup_{t\to\infty}|z_t(t)|>0.
\end{equation}
Then, from (\ref{S4L7E39}), (\ref{S4L7E40}) we can find a sequence of local extremum points $\{t_n\}$ of $z_t$ such that, as $n\to\infty$, one has
\begin{equation}\label{S4L7E41}
z(t_n)\to 0,\quad z_{tt}(t_n)=0,\ \ \textrm{and}\ \ z_t(t_n)\to\ell\neq 0.
\end{equation}
\underline{Case 1}: $z_t\le 0$ for large $t$.\\
Then $\ell<0$.
Passing to the limit in (\ref{S4L4E29}) and using (\ref{S4L7E41}) we find:
\begin{equation}\label{S4L7E42}
-a\ell+(1-L)\ell^2+b\left[1-\left(1-\frac{\ell}{\theta}\right)^{1-\beta}\right]=0,
\end{equation}
however this is a contradiction since $a,b>0$, $\ell<0$ and $\beta\ge 1$.\\
\underline{Case 2}: $z_t\ge 0$ for large $t$.\\
Then $\ell>0$.
Passing to the limit in (\ref{S4L4E29}) we again derive (\ref{S4L7E42}).
Observe that by (\ref{S4L3+E27}) one has $z_t\le\theta$.
Therefore, $\ell\le\theta$.
Also, $\varphi(s)=-as+(1-L)s^2+b\left\{1-(1-s/\theta)^{1-\beta}\right\}$ is decreasing on $(0,\theta]$ since
\begin{align*}
\varphi'(s)=&-a+2(1-L)s+\frac{b(1-\beta)}{\theta}\left(1-\frac{s}{\theta}\right)^{-\beta}\\
\le& -a+2(1-L)s\le -a+2(1-L)\theta=1-\frac{\alpha}{\beta}<0.
\end{align*}
Hnece $\varphi(\ell)<\varphi(0)=0$ which shows that (\ref{S4L7E42}) cannot hold.
This proves that (\ref{S4L7E40}) cannot hold either and completes our argument.
\end{proof}
We next discuss the case where $z_t$ oscillates in a neighborhood of infinity.
\begin{lemma}\label{S4L8}
Assume that $z_t$ changes sign infinitely many times at $t\to\infty$.
Then, there exists a compact set $K$ in the $zz_t$-plane and $t\in\R$ such that
$$
\{(z,z_t);\ t>t_0\}\subset K.
$$
\end{lemma}
\begin{proof}
Note that by (\ref{S4L3+E27}), $z_t$ is bounded.
Also, by Lemma~\ref{S4L3} and (\ref{S4E26}) we have
$$
-\infty<\liminf_{t\to\infty}z(t)<\infty.
$$
To prove our result, it is enough to show that
\begin{equation}\label{S4L8E43}
\limsup_{t\to\infty}z(t)<\infty.
\end{equation}
Assume to the contrary that (\ref{S4L8E43}) does not hold.
Thus, one can find two sequences $\{t_n\}$, $\{\ttn\}$ such that 
\begin{itemize}
\item $t_n$ is a local maximum point for $z$, $z_t(t_n)=0$ and $z(t_n)\to\infty$ as $n\to\infty$.
\item $\ttn>t_n$, $z_t(\ttn)=0$ and $z_t<0$ on $(t_n,\ttn)$.
\end{itemize}
In particular one has that $z(\ttn)<z(t_n)$.
For large $t>0$ set
$$
\Psi(t)=\frac{1}{2}z_t^2+\theta\left(\frac{\alpha}{\beta}-1\right)z+\frac{b}{\beta}e^{-\beta z}.
$$
Then, using (\ref{S4L4E29}) we compute
\begin{align*}
\frac{d\Psi}{dt}
&=z_tz_{tt}+\theta\left(\frac{\alpha}{\beta}-1\right)z_t-be^{-\beta z(t)}z_t\\
&=\left[ a+2\theta\left(L-\frac{g'(u(r))G(u(r))}{g(u(r))^2}\right)\right]z_t^2
+\left(1-\frac{g'G}{g^2}\right)(-z_t)^3\\
&\quad +\theta^2\left(1-\frac{g'G}{g^2}\right)(-z_t)
+b\left[1-\left(1-\frac{z_t}{\theta}\right)^{1-\beta}\right]e^{-\beta z}(-z_t).
\end{align*}
Therefore $\frac{d\Psi}{dt}\ge 0$ on $[t_n,\ttn]$ and $\Psi(t_n)\ge \Psi(\ttn)$.
This yields
\begin{equation}\label{S4L8E44}
\theta\left(\frac{\alpha}{\beta}-1\right)z(\ttn)+\frac{b}{\beta}e^{-\beta z(\ttn)}
\ge \left(\frac{\alpha}{\beta}-1\right)z(t_n)+\frac{b}{\beta}e^{-\beta z(t_n)}
\to\infty\ \ \textrm{as}\ \ t\to\infty.
\end{equation}
Since the function $s\mapsto \theta\left(\frac{\alpha}{\beta}-1\right)s+\frac{b}{\beta}e^{-\beta z(t_n)}$ is increasing in a neighborhood of infinity and $z(\ttn)<z(t_n)$.
It follows from (\ref{S4L8E44}) that $z(\ttn)\to -\infty$ which contradicts $\liminf z_t\in\R$.
\end{proof}

Let us consider the limit problem that corresponds to (\ref{S4L4E29}) that is
\begin{equation}\label{E46}
\begin{cases}
\tz_{tt}-a\tz_t+b(1-e^{-\beta\tz})+(1-L)\tz_t^2+b\left[1-\left(1-\frac{\tz_t}{\theta}\right)^{1-\beta}\right] e^{-\beta\tz}=0,\\
\tz(0)=z_0\in\R
\end{cases}
\end{equation}
which we can write as:
\begin{equation}\label{E47}
\begin{cases}
\tz_t=\tw,\\
\tw_t=a\tw-b(1-e^{-\beta\tz})-(1-L)\tw^2-b\left[1-\left(1-\frac{\tw}{\theta}\right)^{1-\beta}\right] e^{-\beta\tz},\\
(\tz(0),\tw(0))=(\tz_0,\tw_0)\in\R^2.
\end{cases}
\end{equation}

\begin{lemma}\label{S4L9}
Let $(\tz,\tw)$ be the solution of (\ref{E47}) and let $K$ be a compact set in the $\tz\tw$-plane.
If $(\tz_0,\tw_0)\neq (0,0)$ then there exists $T>0$ such that $(\tz(T),\tw(T))\not\in K$.
\end{lemma}
\begin{proof}
The system (\ref{E47}) has the unique equilibrium point $(0,0)$ and the matrix of the linearized system at $(0,0)$ is
$$
\left(
\begin{array}{cc}
0 & 1\\
-b\beta & a+(\beta-1)b
\end{array}
\right).
$$
It is straightforward to compute the two eigenvalues
$$
\lambda_{\pm}=\frac{1}{2}
\left[ a+(\beta-1)b\pm\sqrt{\{a+(\beta-1)b\}^2-4b\beta}\right].
$$
Since $b,\beta>0$, one has ${\rm Re}(\lambda_{\pm})>0$, and hence $(0,0)$ is either an unstable node or a spiral-out.
Thus, the orbit $\{(\tz(t),\tw(t))\}$ does not converge to $(0,0)$.
If $q=1$ then $L=1$ and $\tz$ satisfies
$$
\frac{d}{dt}\left\{\frac{1}{2}\tz_t^2+b\left(\tz+\frac{1}{\beta}e^{-\beta\tz}\right)\right\}
=a\tz_t^2-b\left[ 1-\left(1-\frac{\tz_t}{\theta}\right)^{1-\beta}\right] e^{-\beta\tz} \tz_t\ge 0.
$$
Here, we see that by taking the cases $\tz_t\ge 0$ and $z_t<0$ that
$$
-b\left[1-\left(1-\frac{\tz_t}{\theta}\right)^{1-\beta}\right] e^{-\beta\tz}\tz_t\ge 0.
$$
The last estimate indicates that system (\ref{E47}) has no nontrivial periodic orbit.
If $q>1$, let $\zeta$ be defined by $e^{\tz}=\frac{\zeta^{p+1}}{p+1}$, where $p=q/(q-1)>1$.
Then, from (\ref{E46}) we deduce that $\zeta$ satisfies
$$
\zeta_{tt}-a\zeta_t+\frac{b}{p+1}\left\{\zeta-(p+1)^{\beta}\zeta^{1-\beta(p+1)}\right\}
+b(p+1)^{\beta-1}\zeta^{1-\beta(p+1)}
\left[1-\left(1-\frac{(p+1)\zeta_t}{\theta\zeta}\right)^{1-\beta}\right]=0.
$$
Note that $\beta\ge 1$ and $p>1$ implies $\beta(p+1)>2$.
Consider next
$$
\Phi(t)=\frac{1}{2}\zeta_t^2+\frac{b}{p+1}\left\{\frac{\zeta^2}{2}+\frac{(p+1)^{\beta}}{\beta(p+1)-2}\zeta^{2-\beta(p+1)}\right\}.
$$
Then
$$
\frac{d}{dt}\Phi
=a\zeta_t^2-b(p+1)^{\beta-1}\zeta^{1-\beta(p+1)}\left[1-\left(1-\frac{(p+1)\zeta_t}{\theta\zeta}\right)^{1-\beta}\right]\zeta_t\ge 0
$$
which shows that (\ref{E47}) has no nontrivial periodic orbit.
Thus, (\ref{E47}) has no limit cycle in both cases $q=1$ and $q>1$.

If $\{(\tz(t),\tw(t))\}\in K$ for all $t\ge 0$, then, from the Poincar\'{e}-Bendixon theorem it follows that $(\tz(t),\tw(t))$ either converges to an equilibrium point or approaches a limit cycle which is not possible in light of the above arguments.
Thus, there exists $T>0$ such that $(\tz(T),\tw(T))\not\in K$.
\end{proof}

\begin{lemma}\label{S4L10}
Let $z$ be defined by (\ref{S4E26}).
If $z_t$ changes sign infinitely many times as $t\to\infty$, then
$$
\lim_{t\to\infty}z(t)=\lim_{t\to\infty}z_t(t)=0.
$$
\end{lemma}
\begin{proof}
Suppose by contradiction that
\begin{equation}\label{E48}
\lim_{t\to\infty}(z(t),z_t(t))\neq (0,0).
\end{equation}
By Lemma~\ref{S4L8} there exists a compact set $K\subset\R^2$ and $t_0\ge 0$ so that
\begin{equation}\label{E49}
(z(t),z_t(t))\in K\ \ \textrm{for all}\ t\ge t_0.
\end{equation}
Using (\ref{E48}) and the boundedness of $z$ and $z_t$, we can find a sequence $\{t_n\}\subset\R$ so that $t_n\to\infty$ and
$$
\lim_{n\to\infty}(z(t_n),z_t(t_n))=(z_0,w_0)\neq (0,0).
$$
Let now $(\tz,\tw)$ be the solution of (\ref{E47}) with the initial condition $(z_0,w_0)$.
According to Lemma~\ref{S4L9} there exists $T>0$ such that
\begin{equation}\label{E50}
(\tz(T),\tw(T))\not\in K.
\end{equation}
On the other hand, using (\ref{S1E3}), for large $t>0$ the quantity $\left(L-\frac{g'(u)G(u)}{g(u)^2}\right)(z_t-\theta)^2$ is arbitrarily small since $z_t$ is bounded.
Hence, by the continuous dependence on data for ODEs, for large $n$ we have that $(z(t_n+t),z_t(t_n+t))$ is close to the solution $(\tz(t),\tw(t))$ in the finite interval $[0,T]$.
Using (\ref{E50}) we now deduce that for large $n$, $(z(t_n+t),z_t(t_n+t))\not\in K$ which contradicts (\ref{E49}).
Hence, (\ref{E48}) cannot hold which finishes our proof.
\end{proof}
Now the proof of Proposition~\ref{S4P5} follows from Lemmas~\ref{S4L6}, \ref{S4L7} and \ref{S4L10}.
\begin{corollary}\label{C1}
Let $u$ be a solution of (\ref{S1E1}) and $z$ be defined by (\ref{S4E26}).
Then
$$
\lim_{t\to\infty}z(t)=\lim_{t\to\infty}z_t(t)=0.
$$
In particular, $u(r)=G^{-1}\left[\frac{r^{\theta}}{A}(1+o(1))\right]$ as $r\to 0$.
\end{corollary}

One important tool in the proof of the uniqueness is the following result from \cite{LLD00}.

\begin{proposition} {\rm (see \cite[Lemma~4.2]{LLD00})}\label{P1}
Suppose that $C(t)$ and $D(t)$ are continuous functions satisfying $\lim_{t\to\infty}C(t)=C>0$ and $\lim_{t\to\infty}D(t)=D>0$.
Let $z(t)$ be a solution of 
$$
z_{tt}-C(t)z_t+D(t)z=0\quad \mbox{for large } t.
$$
If $z(t)$ is bounded as $t\to\infty$, then $z(t)\equiv 0$.
\end{proposition}

\begin{theorem}\label{S4T1}
The equation (\ref{S1E1}) has at most one degenerate solution.
\end{theorem}
\begin{proof}
Let $u_j(r)$, $j=1,2$, be two degenerate solutions of (\ref{S1E1}).
Let $z_j(t)$, $j=1,2$, be defined by the transformation \eqref{S4E26}, that is,
$$
e^{z_j(t)}=\frac{r^\theta}{A}G[u_j(r)] \ \ \textrm{and}\ \ t=-\log r.
$$
By \eqref{S4L4E29} we see that $z_j$ satisfies
\[
z_{jtt}-az_{jt}+b(1-e^{-\beta z_j})+(1-L)z_j^2+T(z_j)=0,
\]
where $a$ and $b$ are the positive constants defined in \eqref{S4L4E30} and $T$ is defined in \eqref{S4L4E32} of Lemma \ref{S4L4}. Here and in the subsequent arguments, $z_{jt}$ and $z_{jtt}$, $j=1,2$,  stand for $\frac{dz_j}{dt}$ and $\frac{d^2z_j}{dt^2}$.   
Let $z(t):=z_2(t)-z_1(t)$.
By Corollary~\ref{C1} we see that for $j=1,2$ one has
\begin{equation}\label{T1E1}
z_j(t)\to 0\,, \;\; z_{jt}(t)\to 0 \quad\mbox{ as }\quad t\to\infty.
\end{equation}
Then, $z$ satisfies
\begin{equation}\label{eq1z}
z_{tt}-az_{t}-b\frac{e^{-\beta z_2}-e^{-\beta z_1}}{z_2-z_1}z+(1-L)(z_{2t}+z_{1t}-2\theta)z_t+V(z_2)-V(z_1)+W(z_2)-W(z_1)=0,
\end{equation}
where for $j=1,2$ we defined
\begin{align*}
V(z_j)&=b\Big[1-\Big(1-\frac{z_{jt}}{\theta}\Big)^{1-\beta}\Big]e^{-\beta z_j} \; , \\
W(z_j)&=\Big[L-\frac{g'(u_j)G(u_j)}{g(u_j)^2}\Big](z_{jt}-\theta)^2\; , 
\end{align*}
and  we extended the quotient $(e^{z_2}-e^{z_1})/(z_2-z_1)$ by
\begin{equation}\label{T1E1aa}
\frac{e^{-\beta z_2}-e^{-\beta z_1}}{z_2-z_1}=-\beta e^{-\beta z_2}\quad\mbox{ if } \;\; z_1=z_2.
\end{equation}
Observe first that
\begin{equation}\label{T1E1a}
\begin{aligned}
V(z_2)-V(z_1)=& b\Big[\Big(1-\frac{z_{1t}}{\theta}\Big)^{1-\beta}-\Big(1-\frac{z_{2t}}{\theta}\Big)^{1-\beta}      \Big]e^{-\beta z_1}\\
&+b\left[1-\Big(1-\frac{z_{2t}}{\theta}\Big)^{1-\beta}\right]\cdot \frac{e^{-\beta z_2}-e^{-\beta z_1}}{z_2-z_1}z.
\end{aligned}
\end{equation}
Using the mean value theorem, there exists $\bar z$ between $z_{1t}$ and $z_{2t}$ such that 
\[
b\Big[\Big(1-\frac{z_{1t}}{\theta}\Big)^{1-\beta}-\Big(1-\frac{z_{2t}}{\theta}\Big)^{1-\beta}      \Big]e^{-\beta z_1}=-\frac{b(\beta-1)}{\theta}e^{-\beta z_1}\Big(1-\frac{\bar z}{\theta}\Big)^{-\beta}z_t.
\]
From \eqref{T1E1} we have $\bar z\to 0$ as $t\to \infty$.

Next, we estimate the difference $W(z_2)-W(z_1)$ and we write
\begin{equation}\label{T1E1c}
W(z_2)-W(z_1)=\Big[L-\frac{g'(u_2)G(u_2)}{g(u_2)^2}\Big](z_{1t}+z_{2t}-2\theta)z_t-
\Big[\frac{g'(u_2)G(u_2)}{g(u_2)^2}-\frac{g'(u_1)G(u_2)}{g(u_2)^2}\Big](z_{1t}-\theta)^2.
\end{equation}
To estimate the last term in \eqref{T1E1c}, let us denote $w_j=G[u_j]$, $j=1,2$. 
Since
\[
\frac{d}{dw}\left[\frac{wg'(G^{-1}[w])}{g(G^{-1}[w])^2}\right]=
\left\{1+\left[\frac{g(G^{-1}[w])g''(G^{-1}[w])}{g'(G^{-1}[w])^2}-2\right]
\frac{wg'(G^{-1}[w])}{g(G^{-1}[w])^2}\right\}
\frac{wg'(G^{-1}[w])}{g(G^{-1}[w])^2}\frac{1}{w},
\]
it follows from the mean value theorem that there exists $\bar{w}$ between $w_2$ and $w_1$ such that
\begin{multline*}
\frac{w_2g'(G^{-1}[w_2])}{g(G^{-1}[w_2])^2}-
\frac{w_1g'(G^{-1}[w_1])}{g(G^{-1}[w_1])^2}\\
=\left\{1+\left[\frac{g(G^{-1}[\bar{w}])g''(G^{-1}[\bar{w}])}{g'(G^{-1}[\bar{w}])^2}-2\right]
\frac{\bar{w}g'(G^{-1}[\bar{w}])}{g(G^{-1}[\bar{w}])^2}\right\}
\frac{\bar{w}g'(G^{-1}[\bar{w}])}{g(G^{-1}[\bar{w}])^2}\frac{1}{\bar{w}}(w_2-w_1).
\end{multline*}
Let $w_j=G[u_j]$, $j=1,2$, in the above equality.
Since $G$ is continuous and increasing there exists $\bu$ between $u_2$ and $u_1$ so that $\bar{w}=G[\bu]$ and
\begin{multline}\label{T1E1d}
\frac{g'(u_2)G[u_2]}{g(u_2)^2}-\frac{g'(u_1)G[u_1]}{g(u_1)^2}\\
=\left\{1+\left[\frac{g(\bar{u})g''(\bar{u})}{g'(\bar{u})^2}-2\right]
\frac{g'(\bar{u})G[\bar{u}]}{g(\bar{u})^2}\right\}
\frac{g'(\bar{u})G[\bar{u}]}{g(\bar{u})^2}
\frac{Ae^{-\theta t}}{G[\bar{u}]}
\frac{e^{-\beta z_2}-e^{-\beta z_1}}{z_2-z_1}z,
\end{multline}
where we extended $(e^{-\beta z_2}-e^{-\beta z_1})/(z_2-z_1)$ by \eqref{T1E1aa} if $z_1=z_2$. 
Therefore, from \eqref{eq1z} and \eqref{T1E1a}-\eqref{T1E1d}, $z$ satisfies
\[
z_{tt}-C(t)z_{t}+D(t)z=0,
\]
where
\begin{align*}
C(t)=& \; a-(1-L)(z_{2t}+z_{1t}-2\theta)+\frac{b(\beta-1)}{\theta}e^{-\beta z_1}\Big(1-\frac{\bar z}{\theta}\Big)^{-\beta}\\
&\; -\left[L-\frac{g'(u_2)G[u_2]}{g(u_2)^2}\right](z_{2t}+z_{1t}-2\theta),\\
D(t)=&-b\Big(1-\frac{z_{2t}}{\theta}\Big)^{1-\beta} \cdot \frac{e^{-\beta z_2}-e^{-\beta z_1}}{z_2-z_1}\\
& -\left\{1+\left[\frac{g(\bar{u})g''(\bar{u})}{g'(\bar{u})^2}-2\right]
\frac{g'(\bar{u})G[\bar{u}]}{g(\bar{u})^2}\right\}
\frac{g'(\bar{u})G[\bar{u}]}{g(\bar{u})^2}
\frac{Ae^{-\theta t}}{G[\bar{u}]}
\frac{e^{-\beta z_2}-e^{-\beta z_1}}{z_2-z_1}(z_{1t}-\theta)^2.
\end{align*}
Since $g'(u_2)G[u_2]/g(u_2)^2\to L$, by \eqref{T1E1} we see that
\begin{equation}\label{T1E2}
C(t)\to a+2\theta(1-L)+\frac{b(\beta-1)}{\theta}\ \ \textrm{as}\ \ t\to\infty.
\end{equation}
Since $\bar{u}$ is between $u_2$ and $u_1$, we see that $\bar{u}(r)\to 0$ as $r\to 0$, and hence
\[
1+\left(\frac{g(\bar{u})g''(\bar{u})}{g'(\bar{u})^2}-2\right)
\frac{g'(\bar{u})G[\bar{u}]}{g(\bar{u})^2}
\to 1+\left(\frac{1}{q}-2\right)L =0
\ \ \textrm{as}\ \ r\to 0.
\]
By Corollary~\ref{C1} we see that $G[u_j(r)]=\frac{e^{-\theta t}}{A}(1+o(1))$ as $t\to\infty$.
Since $G$ is monotone, we see that 
$$
G[\bar{u}]=\frac{e^{-\theta t}}{A}(1+o(1))\quad\mbox{ as }\quad t\to\infty.
$$
Then, we find
\begin{equation}\label{T1E3}
D(t)\to b\beta\ \ \textrm{as}\ \ t\to\infty.
\end{equation}
Because of (\ref{T1E1}), (\ref{T1E2}) and (\ref{T1E3}), by Proposition~\ref{P1} we see that $z(t)\equiv 0$, and hence the conclusion holds.
\end{proof}

\begin{proof}[Proof of Theorem~\ref{THA} (i)-(iii).]

(i) The existence of a solution was already proved in Section~3.
The uniqueness of a solution to (\ref{S1E1}) follows from Theorem~\ref{S4T1}.

(ii)
By Corollary \ref{C1} we have
$$
u(r)=G^{-1}\left[\frac{r^{\theta}}{A}\psi(r)\right],
$$
where $0<\psi\in C[0,r_0)$ and $\psi(r)\to 1$ as $r\to 0$. 
Let $s=G^{-1}[A^{-1}r^{\theta}\psi(r)]$, so that $s\to 0$ as $r\to 0$ and
$r=\left(\frac{A}{\psi(r)}\right)^{1/\theta}G(s)^{1/\theta}$.
Then
\begin{multline}\label{S4Lim1}
u'(0)=\lim_{r\to 0}\frac{u(r)}{r}
=\lim_{r\to 0}\frac{G^{-1}[A^{-1}r^{\theta}(\psi(r))]}{r}
=\lim_{r\to 0}\frac{\psi(r)^{1/\theta}s}{A^{1/\theta}G(s)^{1/\theta}}
=\left[\frac{1}{A} \lim_{s\to 0}\frac{ s^{\theta}}{G(s)}\right]^{1/\theta}
\end{multline}
which proves \eqref{S1E8}.

Assume next that \eqref{S1E9} holds and take $\e>0$ so that $\frac{\gamma-\alpha}{\beta}<\frac{q+\e}{q+\e-1}$.
By condition $(G2)$ it follows  
$$
\frac{g'(s)^2}{g(s)g''(s)} \le q+\e \quad\mbox{ for all }s>0 \mbox{ small}.
$$ 
From here we see that 
$$
\left( \frac{g}{g'} \right)'(s)\leq \frac{q+\e-1}{q+\e}\quad\mbox{ for all }s>0 \mbox{ small}.
$$
Integrating in the above estimate we find that $s\longmapsto g(s)s^{-\frac{q+\e}{q+\e-1}}$ is increasing in a small neighborhood of the origin.  Thus, there exists $c>0$ such that 
$1/g(s)\ge cs^{-\frac{q+\e}{q+\e-1}}$  for  all $s>0$  small.
This further yields
$$
\frac{s^{\theta-1}}{g(s)}\ge c s^{\theta-1-\frac{q+\e}{q+\e-1}}\to\infty
\ \ \textrm{as}\ s\to 0,
$$
since, by our choice of $\e>0$ we have  $\theta-1-\frac{q+\e}{q+\e-1}<0$. By L'Hospital's rule it follows that $s^{\theta}/G(s)\to \infty$ as $s\to 0$ and then, by \eqref{S4Lim1} we have $u'(0)=\infty$.

(iii) The convergence part follows from the uniqueness of a solution to \eqref{S1E1} and \eqref{S3L3E20}.
\end{proof}

\section{Proof of Theorem~\ref{THA}~(iv)}
Let $u(r,a)$ be a solution of (\ref{S1E4}).
We define
\begin{equation}\label{S5E0}
\tu(s)=G_q^{-1}\left[\lambda^{-\theta}G\left[u(r,a)\right]\right],\quad s=\frac{r}{\lambda}
\ \ \textrm{and}\ \ \lambda=\left(\frac{G[a]}{G_q[1]}\right)^{1/\theta}.
\end{equation}
Then $\tu(s)$ satisfies
\begin{equation}\label{S5E1}
\begin{cases}
s^{-(\gamma-1)}(s^{\alpha}\tu'(s)^{\beta})'-\frac{1}{g_q(\tu(s))^{\beta}}
+s^{-\gamma+\alpha-1}\frac{g_q(\tu(s))}{G_q(\tu(s))}\left(L-\frac{g'(u(r))G(u(r))}{g(u(r))^2}\right)\tu'(s)^{\beta+1}=0, & s>0,\\
\tu(0)=1,\\
\tu(0)=0.
\end{cases}
\end{equation}

\begin{lemma}\label{S5L1}
Let $v(s,1)$ be a solution of (\ref{S1E10}) with $b=1$.
Then,
\[
\tu(s)\to v(s,1)\ \ \textrm{in}\ \ C_{\rm{loc}}[0,\infty)\ \ \textrm{as}\ \ a\to 0.
\]
\end{lemma}
\begin{proof}
Let $s_0>0$ be fixed. We claim that
\begin{equation}\label{S5L1E1}
u(\lambda s,a)\to 0\ \ \textrm{uniformly in}\ \ s\in[0,s_0]\ \ \textrm{as}\ \ a\to 0.
\end{equation}
Indeed, let $\delta>0$ be small.
By Lemma~\ref{S3L3}~(ii) there exist $r_{\delta}>0$ and $\delta_0>0$ such that
\begin{equation}\label{S5L1E2}
\textrm{if}\ 0<a<\delta_0,\ \textrm{then}\ 0\le u(r,a)<\delta\ \textrm{for all}\ 0<r<r_{\delta}.
\end{equation}
If $a>0$ is small, then $s_0<r_{\delta}/\lambda$, because $\lim_{a\to 0}\lambda=\lim_{a\to 0}\left(G[a]/G_q[1]\right)^{1/\theta}=0$.
By (\ref{S5L1E2}) we see that if $a>0$ is small, then $0\le u(\lambda s,a)<\delta$ for $0\le s\le s_0$.
Since $\delta>0$ can be chosen arbitrarily small, we see that (\ref{S5L1E1}) follows.

By (\ref{S5L1E1}) we have
\[
\frac{g'(u(\lambda s,a))G[u(\lambda s,a)]}{g(u(\lambda s,a))^2}\to L
\ \ \textrm{uniformly in}\ \ s\in[0,s_0]\ \ \textrm{as}\ \ a\to 0.
\]
Clearly $G_q[\tu]\ge G_q[1]>0$ and the denominator $G_q[\tu]$ in (\ref{S5E1}) is uniformly bounded away from $0$.
Because of the continuity of $\tu(s)\in C[0,s_0)$ with respect to the nonlinearity in (\ref{S5E1}), we see that $\tu(s)\to v(s,1)$ in $C[0,s_0)$ as $a\to 0$.
Since $s_0>0$ can be chosen arbitrarily large, the conclusion follows.
\end{proof}

\begin{lemma}\label{S5L2}
Let $u^*(r)$ be the degenerate solution given by Theorem~\ref{THA}, and let $s$ and $\lambda$ be defined by (\ref{S5E0}).
Let $\tu^*(s):=G^{-1}_q\left[\lambda^{-\theta}G[u^*(r)]\right]$. Then
\[
\tu^*(s)\to v^*(s)\ \ \textrm{in}\ \ C_{\rm{loc}}(0,\infty)\ \ \textrm{as}\ \ a\to 0,
\]
where $v^*(s)=G_q^{-1}[A^{-1}s^{\theta}]$ which is defined by (\ref{S1E11}).
\end{lemma}
\begin{proof}
There exists a continuous function $\rho (r)$ such that $u^*(r)=G^{-1}[A^{-1} r^{\theta}(1+\rho (r))]$ and $\rho(r)\to 0$ as $r\to 0$.
We have
\[
\tu^*(s)=G^{-1}_q\left[\lambda^{-\theta}G\left[G^{-1}[A^{-1} r^{\theta}(1+\rho(\lambda s))]\right]\right]
=G_q^{-1}\left[A^{-1} s^{\theta}(1+\rho(\lambda s))\right].
\]
Let $0<s_0<s_1$ be fixed.
Since $\rho(r)\to 0$ as $r\to 0$, we see that $\rho(\lambda s)\to 0$ uniformly in $s\in[s_0,s_1]$ as $a\to 0$.
Thus, $\tu^*(s)\to G_q^{-1}[A^{-1} s^{\theta}]$ uniformly in $s\in [s_0,s_1]$ as $a\to 0$.
Since $s_0$ and $s_1$ can be chosen arbitrary, the conclusion follows.
\end{proof}

In the proof of Theorem~\ref{THA}~(iv) we use the following:

\begin{proposition}\label{S1P1}
Let $q_c$ be defined by (\ref{qc}).
Assume that $b>0$ (resp. $b\in\R$) if $q>1$ (resp. if $q=1$).
Let $v(s,b)$ be a solution of (\ref{S1E10}), and let $v^*(s)$ be the 
solution given by (\ref{S1E11}).
If $q<q_c$, then $\calZ_{(0,\infty)}[v(\,\cdot\,,b)-v^*(\,\cdot\,)]=\infty$.
\end{proposition}

\begin{proof}[Proof of Theorem~\ref{THA}~(iv)]
Let
\[
\tu(s):=G^{-1}_q[\lambda^{-\theta}G[u(r,a)]],\ \tu^*(s):=G_q^{-1}[\lambda^{-\theta}G[u^*(r)]],\ s:=\frac{r}{\lambda}
\ \textrm{and}\ \lambda:=\left(\frac{G[a]}{G_q[1]}\right)^{1/\theta}.
\]
By Lemmas~\ref{S5L1} and \ref{S5L2} we see that
\begin{equation}\label{PA2E1}
\tu(s)\to v(s,1)\ \ \textrm{in}\ \ C_{\rm{loc}}(0,\infty)\ \ \textrm{as}\ \ a\to 0,
\end{equation}
\begin{equation}\label{PA2E2}
\tu^*(s)\to v^*(s)\ \ \textrm{in}\ \ C_{\rm{loc}}(0,\infty)\ \ \textrm{as}\ \ a\to 0.
\end{equation}
By Proposition~\ref{S1P1} we have
\begin{equation}\label{PA2E3}
\calZ_{(0,\infty)}\left[v(\,\cdot\,,1)-v^*(\,\cdot\,)\right]=\infty.
\end{equation}
Let $r_0>0$ be fixed.
Since the same transformation is applied to both $\tu(s)$ and $\tu^*(s)$, we have
$\calZ_{(0,r_0)}[u(\,\cdot\,,a)-u^*(\,\cdot\,)]=\calZ_{(0,r_0/\lambda)}[\tu(\,\cdot\,)-\tu^*(\,\cdot\,)]$.
For each $M>0$, there are $s_M>0$ and $a_M>0$ such that $\calZ_{(0,s_M)}[\tu(\,\cdot\,)-\tu^*(\,\cdot\,)]\ge M$ for $0<a<a_M$, because of (\ref{PA2E1}), (\ref{PA2E2}) and (\ref{PA2E3}).
If $a>0$ is small, then $(0,s_M)\subset (0,r_0/\lambda)$, and hence
\[
\calZ_{(0,r_0)}[u(\,\cdot\,,a)-u^*(\,\cdot\,)]=\calZ_{(0,r_0/\lambda)}[\tu(\,\cdot\,)-\tu^*(\,\cdot\,)]\ge
\calZ_{(0,s_M)}[\tu(\,\cdot\,)-\tu^*(\,\cdot\,)]\ge M.
\]
Since $M$ can be arbitrarily large, we see that $\calZ_{(0,r_0)}[u(\,\cdot\,,a)-u^*(\,\cdot\,)]\to\infty$ as $a\to 0$.
\end{proof}

\section{Bifurcation diagram}
Let $(\lambda,v(s))$ be a regular solution of (\ref{v}), and let $\Lambda:=\lambda^{1/(\gamma-\alpha+\beta)}$.
Set $u(r):=1-v(s)$ and $r:=\Lambda s$.
Then $u$ satisfies
\[
\begin{cases}
r^{-(\gamma-1)}(r^{\alpha}|u'|^{\beta-1}u')'=\frac{1}{f(u)} & \textrm{for}\ 0<r<\Lambda,\\
0\le u(r)<1 & \textrm{for}\ 0\le r<\Lambda,\\
u(\Lambda)=1.
\end{cases}
\]

\begin{proof}[Proof of Corollary~\ref{C0}~(i)]
By Theorem~\ref{THA}~(i) we have that (\ref{S1E1}) has a unique degenerate solution $u^*(r)$. Also, by Lemma~\ref{S4L1} (i)-(ii) we have
\begin{equation}\label{C0E00}
{u^*}'(r)>0\quad\mbox{ and }\quad u^*(r)\ge G^{-1}\left[\frac{r^{\theta}}{\theta\gamma^{1/\beta}}\right].
\end{equation} 
Using the assumption (G4) (which yields $G(\infty)=\infty$),  there exists a unique $r^*_0>0$ such that $u^*(r^*_0)=1$.
Then $(\bar{\lambda},\bar{v}(s)):=((r^*_0)^{\gamma-\alpha+\beta},1-u^*(r^*_0s))$ is a degenerate solution of (\ref{v}).
Let us prove the uniqueness.
Suppose that there are two degenerate solutions $(\lambda_j,v_j(s))$, $j=1,2$, of (\ref{v}).
Then the two functions $u_j(r):=1-v_j(\Lambda_j^{-1}r)$, $j=1,2$, are  solutions of (\ref{S1E1}), where $\Lambda_j:=\lambda_j^{1/(\gamma-\alpha+\beta)}$.
Because of the uniqueness of the solution $u^*$ of (\ref{S1E1}), we see that $u^*(r)=1-v_j(\Lambda_j^{-1}r)$.
Since $v_j(s)$ satisfies the Dirichlet boundary condition, we see that $0=v_j(1)=1-u^*({\Lambda_j})$.
By the uniqueness of $r^*_0$ we have that ${\Lambda_1}={\Lambda_2}=r^*_0$.
Since $u^*(r)=1-v_j(\Lambda_j^{-1}r)$, we see that $v_1(s)=v_2(s)$ which shows the uniqueness of a degenerate solution to \eqref{v}.

We next establish (\ref{C0E0}).
Let $\delta>0$ be fixed. By \eqref{C0E00} and assumption (G4) there exists $r_{\delta}^*>0$ such that $u^*(r_{\delta}^*)=1+\delta$. For $\tau\in (0,1)$, let $u(\cdot, 1-\tau)$ be the solution of (\ref{S1E4}) with $a=1-\tau$. 
Let $r_0(\tau)>0$ be the first positive zero of the function $1-u(\,\cdot\,,1-\tau)$;
note that $r_0(\tau)$ exists due to the convergence (\ref{S1Econv}).
Then 
$$
(\lambda(\tau),v(s,\tau)):=\big(r_0(\tau)^{\gamma-\alpha+\beta},1-u(r_0(\tau)s,1-\tau)\big)
$$ 
is a solution of (\ref{v}).
If $\tau$ is close to $1$, then the solution $u(\,\cdot\,,1-\tau)$ exists in $[0,r_{\delta}^*]$ and
\begin{equation}\label{C0E1}
u(r,1-\tau)\to u^*(r)\ \ \textrm{in}\ \ C[0,r_{\delta}^*]\ \ \textrm{and}\ \ \tau\to 1.
\end{equation}
Since $1-u^*(r_{\delta}^*)<0$, we see that $r_0(\tau)<r_{\delta}^*$ provided that $\tau$ is close to $1$.
By (\ref{C0E1}) we have $r_0(\tau)\to r_0^*$ as $\tau\to 1$.
Since $\bar{\lambda}=(r_0^*)^{\gamma-\alpha+\beta}$ and $\lambda(\tau)=r_0(\tau)^{\gamma-\alpha+\beta}$, we see that $\lambda(\tau)\to\bar{\lambda}$ as $\tau\to 1$ and thus
\begin{equation}\label{C0E2}
|u^*(\bar{\lambda}^{1/(\gamma-\alpha+\beta)}s)-u^*(\lambda(\tau)^{1/(\gamma-\alpha+\beta)}s)|\to 0\ \ \textrm{in}\ \ C[0,1]\ \ \textrm{as}\ \ \tau\to 1.
\end{equation}
By Theorem~\ref{THA}~(i) we have
\begin{equation}\label{C0E3}
|u(\lambda(\tau)^{1/(\gamma-\alpha+\beta)}s,1-\tau)-u^*(\lambda(\tau)^{1/(\gamma-\alpha+\beta)}s)|\to 0\ \ \textrm{in}\ \ C[0,1]\ \ \textrm{as}\ \ \tau\to 1.
\end{equation}
By (\ref{C0E2}) and (\ref{C0E3}) we have
\begin{align*}
|v(s,\tau)-\bar{v}(s)|
&\le|u^*(\bar{\lambda}^{1/(\gamma-\alpha+\beta)}s)-u^*(\lambda(\tau)^{1/(\gamma-\alpha+\beta)}s)|\\
&\quad +|u^*(\lambda(\tau)^{1/(\gamma-\alpha+\beta)}s)-u(\lambda(\tau)^{1/(\gamma-\alpha+\beta)}s,1-\tau)|\\
&\to 0\ \ \textrm{in}\ \ C[0,1]\ \ \textrm{as}\ \ \tau\to 1.
\end{align*}
This finishes our proof. 
\end{proof}

\begin{proof}[Proof of Corollary~\ref{C0}~(ii)]
The proof is inspired from \cite[Lemma~8.1]{M18a}.
However, the situation is diffrent here, since  solutions to (\ref{v}) may be degenerate.

Let $U(r,a):=v(s,a)$ and $r:=\lambda^{1/(\gamma-\alpha+\beta)}s$.
Then, $U$ satisfies
\begin{equation}\label{C0PE1}
\begin{cases}
r^{-(\gamma-1)}(r^{\alpha}|U'|^{\beta-1}U')'+\frac{1}{f(1-U)}=0 & 0<r<\lambda^{1/(\gamma-\alpha+\beta)},\\
U(0,a)=1-a,\\
U(r_0(a),a)=0.
\end{cases}
\end{equation}
Let $U^*(r)$ denote the unique degenerate solution of the equation in (\ref{C0PE1}) such that $U^*(0)=1$ and $(U^*)'(r)<0$ for $r>0$.
The existence of $U^*$ is guaranteed by Theorem~\ref{THA}~(i).
Let $r_0(a)$ and $r_0^*$ denote the first zero of $U(\,\cdot\,,a)$ and $U^*(\,\cdot\,)$, respectively. 
Then, $\lambda(a)$ and $\bar{\lambda}$ are given by
\[
\lambda(a):=r_0(a)^{\gamma-\alpha+\beta}\quad\textrm{and}\quad\bar{\lambda}:=(r_0^*)^{\gamma-\alpha+\beta}.
\]

Let $I:=[0,\min\{r_0(a),r_0^*\}]$ and set
$z(a):=\calZ_I[U(\,\cdot\,,a)-U^*(\,\cdot\,)]$.
For each fixed $a\in (0,1)$, there is a neighborhood of $r=0$ such that $U(\,\cdot\,,a)$ and $U^*(\,\cdot\,)$ has no intersection in the heighborhood, since $U(0,a)<1=U^*(0)$.
The equation in (\ref{C0PE1}) is equivalent to 
$$
\beta U''+\frac{\alpha}{r}U'+\frac{{r^{\gamma-\alpha-1}}}{f(1-U)|U'|^{\beta-1}}=0.
$$
Since $(U^*)'<0$ for $r>0$ by Lemma~\ref{S4L1} and $U'<0$ for $r>0$ by Lemma~\ref{S3L2}, the equation (\ref{C0PE1}) is nondegenerate outside the neighborhood of $r=0$.
Therefore, for each fixed $a\in (0,1)$, the zero set $\{U(\,\cdot\,,a)-U^*(\,\cdot\,)=0\}$ does not have an accumulation point, and hence $z(a)<\infty$.
It is clear that $U(r,a)-U^*(r)$ is a $C^1$ function of $(r,a)$.
We see that each zero of $U(\,\cdot\,,a)-U^*(\,\cdot\,)$ is simple, because of the uniqueness of a solution of the ODE of second order
\begin{multline*}
\beta(U-U^*)''+\frac{\alpha}{r}(U-U^*)'
-\frac{r^{\gamma-\alpha-1}}{|U'|^{\beta-1}}\frac{1}{f(1-U)f(1-U^*)}\frac{f(1-U)-f(1-U^*)}{U-U^*}(U-U^*)\\
-\frac{r^{\gamma-\alpha-1}}{f(1-U^*)|U'|^{\beta-1}|(U^*)'|^{\beta-1}}\frac{|U'|^{\beta-1}-|(U^*)'|^{\beta-1}}{U'-(U^*)'}(U'-(U^*)')=0.
\end{multline*}
It follows from the implicit function theorem that each zero of $U(\,\cdot\,,a)-U^*(\,\cdot\,)$ continuously depends on $a$.
Because $z(a)$ does not change in a neighborhood of each fixed $a$, $z(a)$ is conserved if another zero does not enter $I$ from $\partial I$ or goes out of $I$.
Here, the boundedness of $z(a)$ is used to guarantee the local conservation of $z(a)$.
We prove the conclusion by contradiction.
Suppose that there is $0<a_0<1$ such that $\lambda(a)<\bar{\lambda}$ for all $a\in(a_0,1)$.
Let $\tilde{r}:=\min\{r_0(a),r_0^*\}$.
Since $U(0,a)-U^*(0)<0$, it follows that $U(\tilde{r},a)-U^*({\tilde{r}})<0$.
Thus, another zero cannot enter or go out.
Hence $z(a)$ is bounded for $a\in (a_0,1)$ which
contradicts Theorem~\ref{THA}~(iv).
Similarly, we obtain a contradiction in the case where $\lambda(a)>\bar{\lambda}$ for $a\in (a_0,1)$.
Thus, $\lambda(a)$ has to oscillate infinitely many times around $\bar{\lambda}$ as $a\to 1$.
\end{proof}



\begin{thebibliography}{99}
\bibitem{CFT11} {D.~Cassani, L.~Fatorusso and A.~Tarsi},  {\it Global existence for nonlocal MEMS}, 
Nonlinear Anal. {\bf 74} (2011), 5722--5726.

\bibitem{CES08} {D.~Castorina, P.~Esposito and B.~Sciunzi}, $p$-MEMS equation on a ball, 
Methods Appl. Anal.  {\bf 15} (2008), 277--284.

\bibitem{D00}{E.~Dancer},
{\it Infinitely many turning points for some supercritical problems},
Ann. Mat. Pura Appl. {\bf 178} (2000), 225--233.

\bibitem{DW12} {J.~Davila and J.~Wei}, {\it Point ruptures for a MEMS equation with fringing field}, Comm. Part. Differential Equations {\bf 37} (2012), 1462--1493.

\bibitem{DWW16}{J.~Davila, K.~Wang and J.~Wei},
{\it Qualitative analysis of rupture solutions for a MEMS problem},
Ann. Inst. H. Poincar\'{e}, Anal. Non Lin\'{e}aire {\bf 33} (2016), 221--242.
 
\bibitem{dd17}{J.~do \'{O} and E.~da Silva},
{\it Some results for a class of quasilinear elliptic equations with singular nonlinearity},
Nonlinear Anal. {\bf 148} (2017), 1--29. 

\bibitem{EG08} {P. ~Esposito and N.~Ghoussoub},  {\it  Uniqueness of solutions for an elliptic equation modeling MEMS}, Methods Appl. Anal.  {\bf 15} (2008), 341--354.

\bibitem{EGG10}{P.~Esposito, N.~Ghoussoub and Y.~Guo},
{Mathematical Analysis of Partial Differential Equations Modelling Electrostatic MEMS},
{\it pp. 1-318, Courant Lecture Notes in Maths 20 (2010). CIMS/AMS.}

\bibitem{ES18} C. Esteve and Ph. Souplet, {\it Quantitative touchdown localization for the MEMS problem with variable dielectric permittivity}, Nonlinearity {\bf 31} (2018), 4883--4934.

\bibitem{ES19} C. Esteve and Ph. Souplet, {\it No touchdown at points of small permittivity and nontrivial touchdown sets for the MEMS problem},  Adv. Differential Equations {\bf 24} (2019), 465--500.

\bibitem{GG20}{M. Ghergu and O. Goubet},
{\it Singular solutions of elliptic equations with iterated exponentials},
J. Geometric Anal., {\bf 30} (2020), 1755--1773.


\bibitem{GS15} {J.-S. Guo and Ph. Souplet}, {\it No touchdown at zero points of the permittivity profile for the MEMS problem}, SIAM J. Math. Anal. {\bf 47} (2015), 614--625.

\bibitem{GLWZ11}{Z.~Guo, Z.~Liu, J.~Wei and F.~Zhou},
{\it Bifurcations of some elliptic problems with a singular nonlinearity via Morse index},
Commun. Pure Appl. Anal. {\bf 10} (2011), 507--525. 


\bibitem{GW08a}{Z.~Guo and J.~Wei},
{\it Infinitely many turning points for an elliptic problem with a singular non-linearity},
J. Lond. Math. Soc. {\bf 78} (2008), 21--35. 

\bibitem{GW08b}{Z.~Guo and J.~Wei},
{\it On solutions with point ruptures for a semilinear elliptic problem with singularity},
Methods Appl. Anal. {\bf 15} (2008), 377--390.

\bibitem{GW08c} {Z.~Guo and J.~Wei}, {\it Asymptotic Behavior of touch-down solutions and global bifurcations for an elliptic problem with a singular nonlinearity}, Comm. Pure Appl. Anal. {\bf 7}(2008), 765--786.

\bibitem{GW14} {Z.~Guo and J.~Wei}, {\it Rupture solutions of an elliptic equation with a singular nonlinearity}, Proc. Roy. Soc. Edin., A {\bf 144} (2014),  905--924.

\bibitem{GPW05}{Y. Guo, Z. Pan and M.J. Ward},  {\it Touchdown and pull-in voltage behavior of a MEMS device with varying dielectric properties}, 
SIAM J. Appl. Math. {\bf 66} (2005), 309--338.


\bibitem{K97}{P.~Korman},
{\it Solution curves for semilinear equations on a ball},
Proc. Amer. Math. Soc. {\bf 125} (1997), 1997--2005.

\bibitem{K18}{P.~Korman},
{\it Infinitely many solutions for three classes of self-similar equations with p-Laplace operator: Gelfand, Joseph-Lundgren and MEMS problems},
Proc. Roy. Soc. Edinburgh Sect. A {\bf 148} (2018), 341--356.

\bibitem{LW14} {Ph.~Lauren\c cot and C.~Walker}, {\it A fourth-order model for MEMS with clamped boundary conditions}, Proc. London Math. Soc. {\it 109} (2014), 1435--1464.

\bibitem{LW17} {Ph.~Lauren\c cot and C.~Walker}, {\it Some singular equations modelling MEMS},
Bulletin AMS {\bf 54} (2017), 437--479.

\bibitem{LGG05} {K.~Li, H.~Guo and Z.~Guo}, {\it Positive single rupture solutions to a semilinear elliptic equation}, Appl. Math. Letters {\bf 18} (2005), 1177--1183.

\bibitem{LW11}{A.~Lindsay and M.~Ward},
{\it Asymptotics of some nonlinear eigenvalue problems modelling a MEMS capacitor. Part II: multiple solutions and singular asymptotics},
European J. Appl. Math. {\bf 22} (2011), 83--123.

\bibitem{LLD00}{Y. Liu, Y. Li and Y. Deng},
{\it Separation property of solutions for a semilinear elliptic equation},
J. Differential Equations {\bf 163} (2000), 381--406.


\bibitem{M16}{Y.~Miyamoto}, {\it Intersection properties of radial solutions and
global bifurcation diagrams for supercritical quasilinear elliptic equations}, Nonlinear Differential Equations Appl. NoDEA, (2016) 23:16.

\bibitem{M18a}{Y.~Miyamoto},
{\it A limit equation and bifurcation diagrams of semilinear elliptic equations with general supercritical growth},
J. Differential Equations {\bf 264} (2018), 2684--2707.

\bibitem{MN18}{Y. Miyamoto and Y. Naito},
{\it Singular extremal solutions for supercritical elliptic equations in a ball},
J. Differential Equations {\bf 265} (2018), 2842--2885.

\bibitem{MN20}{Y.~Miyamoto and Y. Naito},
{\it Fundamental properties and asymptotic shapes of the singular and classical radial solutions for supercritical semilinear elliptic equations},
preprint.

\bibitem{PB03}{J.A.~Pelesko and D.H.~Bernstein},
Modelling MEMS and NEMS,
Chapman \& Hall CRC, Boca Raton, FL, 2003.


\end{thebibliography}
\end{document}